\documentclass[journal]{IEEEtran}
\usepackage[utf8]{inputenc}
\usepackage{mathtools}
\mathtoolsset{showonlyrefs,showmanualtags}
\usepackage{amssymb,amsthm}
\usepackage{dsfont}
\usepackage{booktabs}
\usepackage{color}
\usepackage{bbold} 
\usepackage[final]{graphicx}
\usepackage[yyyymmdd,hhmmss]{datetime}
\usepackage{bm,upgreek}
\usepackage{tablefootnote}
\usepackage{nth}

\DeclareMathOperator{\Binopdf}{binopdf}
\DeclareMathOperator{\MeanSpace}{\mathcal{E}}
\DeclareMathOperator{\DeltaS}{\mathcal{P}_{\scriptscriptstyle{\mathcal{S}}}}
\DeclareMathOperator{\QS}{\mathcal{Q}_{\scriptscriptstyle{\mathcal{S}^c}}}
\DeclareMathOperator{\QR}{\mathcal{Q}_{\scriptscriptstyle{\mathcal{R}}}}
\DeclareMathOperator{\IX}{\mathcal{I}}

\DeclareMathOperator{\IXR}{\mathcal{I}_{\scriptscriptstyle{\mathcal{R}}}}
\newcommand{\BigO}{\mathcal{O}(\frac{1}{\sqrt n})}
\newcommand{\BigON}{ \mathcal{O}(\tfrac{1}{\sqrt{ |\mathcal{N}^k|}})}
\newcommand{\Emp}[1]{\xi\left( #1\right)}
\newcommand{\Exp}[1]{\mathbb{E}[ #1]}
\newcommand{\Ninf}[1]{ \| #1 \| }

\newcommand{\ID}[1]{ \mathbb{1} (#1 ) }
\newcommand{\Prob}[1]{\mathbb{P} (#1)}    
\newcommand{\ProbK}[1]{\mathbb{P}^k (#1)}     
\newcommand{\argmin}{\operatornamewithlimits{argmin}} 
\newcommand{\Bz}{\boldsymbol{\mathfrak{Z}}} 
\newcommand{\LipK}[1]{H^{k,#1}} 
\newcommand{\Lip}[1]{H^{#1}} 
\newcommand{\FS}{\hat f} 
\newcommand{\qss}{r} 
\newcommand{\LS}{\ell} 
 
\newcommand{\VP}{V^p} 
\newcommand{\VM}{V^d} 
\newcommand{\VQ}{V^q} 
\newcommand{\VPQ}{V^{pq}} 
\newcommand{\PW}{\mathbb{P}_{W^k_t}(w) } 
\newcommand{\NT}{t \in \mathbb{N}_T} 
\newcommand{\MW}{\underline{w}_t^k} 
\newtheorem{Assumption}{Assumption}
\newtheorem{Theorem}{Theorem}
\newtheorem{Definition}{Definition}
\newtheorem{Remark}{Remark}
\newtheorem{Lemma}{Lemma}
\newtheorem{Corollary}{Corollary}
\newtheorem{Proposition}{Proposition}
\newtheorem{Problem}{Problem}

\newcommand\SUMI{\sum_{i\in\mathcal{N}^k}}
\newcommand{\edit}[1]{\textcolor{black}{#1}}
\newcommand{\Compress}{\medmuskip=0mu
\thinmuskip=0mu
\thickmuskip=0mu}

\begin{document}

\title{Deep Teams: Decentralized Decision Making   with Finite and Infinite Number of Agents}
\author{Jalal Arabneydi,~\IEEEmembership{Member,~IEEE,} and Amir G. Aghdam,~\IEEEmembership{Senior~Member,~IEEE}
\thanks{ This work has been supported in part by the Natural Sciences and Engineering Research Council of Canada (NSERC) under Grant RGPIN-262127-17, and in part by Concordia University under Horizon Postdoctoral Fellowship.}  
\thanks{Jalal Arabneydi and Amir G. Aghdam are with the  Department of Electrical and Computer Engineering, 
        Concordia University, 1455 de Maisonneuve Blvd. West, Montreal, QC, Canada, Postal Code: H3G 1M8.
        {\tt\small Email:jalal.arabneydi@mail.mcgill.ca} and
        {\tt\small Email:aghdam@ece.concordia.ca}}%
              
}

\maketitle

\begin{abstract}
Inspired by the concepts of  deep learning in artificial intelligence and fairness in behavioural economics,  we introduce  deep  teams in this paper.  In such systems, agents  are partitioned into a few  sub-populations so that  the dynamics and cost of  agents in each sub-population  is  invariant to the  indexing of agents. The goal of agents is to minimize a common cost function in such a manner that  the agents in each sub-population are not discriminated or privileged by the way they are indexed. Two non-classical information structures are studied. In the first one, each agent observes  its local state as well as the empirical distribution of the states of agents in each sub-population, called deep state, whereas in the second one,  the deep states of a subset (possibly all) of sub-populations are not observed. Novel  dynamic programs are developed to identify  globally optimal and sub-optimal solutions for the first and second information structures, respectively. The computational complexity of finding the optimal solution in both space and time is polynomial (rather than exponential) with respect to the number of agents in each sub-population and is linear (rather than exponential) with respect to the control horizon.  This complexity is further reduced  in time by introducing a  forward equation, that we  call   deep Chapman-Kolmogorov equation,  described by multiple  convolutional layers  of  Binomial probability distributions.  Two different prices are defined  for computation and communication, and it is shown that under mild conditions they converge to zero as the number of  quantization levels and the number of agents tend to infinity. In addition,  the main results are  extended to   infinite-horizon discounted  models and arbitrarily asymmetric cost functions. Finally,  a service-management example  with 200 users is presented.
\end{abstract}


\section{Introduction}\label{sec:introduction}
 Team theory studies  cooperative decision making and   is used   in numerous   applications such as smart grids,  swarm robotics, transportation networks, social networks, and emergent behaviours, to name only a few.   Such applications often consist of a group of interconnected decision makers, modelled as Markov decision processes,  that  wish to accomplish a common task  in the presence of  limited computation and communication  resources.  Historically,  team theory can be traced back to the seminal work of  Radner~\cite{Radner1962} and Marschack and Radner~\cite{MarschackRadner1972} on static teams as well as Witsenhausen~\cite{Witsenhausen1971separation} and Ho~\cite{ho1980team} on dynamic teams.  For a comprehensive literature overview,  the interested  reader is referred to~\cite{Yuksel2013stochastic} and references therein.

   When centralized information structure  is feasible (i.e., joint state is known to all decision makers),   optimal solution  is given by the celebrated dynamic programming principle~\cite{kumar2015stochastic}.  In this case, the  number of  computational resources (in space and time)\footnote{In this paper,  the computational resources in  space  and time,   respectively, refer to the size of memory and  the number of iterations that  an algorithm needs to be run in order to perform a task.  It is to be noted that the  computational complexity  in  time  is different from  the computational complexity in control horizon.} to identify the optimal solution increases  exponentially with the number of decision makers, in general; a phenomenon  known as the ``curse of dimensionality''. For example,   a centralized system consisting of $100$  decision makers   with binary states  requires   the computational resources of  order~$2^{100} \approx 10^{30}$,  which is intractable.  On the other hand,  centralized  information  structure may  not even be feasible  due to   limited communication resources,  specially when the number of decision makers is large.   In such cases,  it is desired to have some form of decentralized  information structure. However,  a lower communication requirement comes at the cost of a harder optimization problem to solve   because  the decision makers in a decentralized structure may  have different perception about the system~\cite{Bernstein2002complexity}.
Due to the high complexity of the decentralized control problems,  an explicit optimal solution may be intractable for systems with more than  two or three decision makers  \cite{ouyang2015signaling,lessard2015optimal}.  As a result,  after  nearly 60 years of research,  there is still a big gap between theory and practice.

In this paper,  inspired by the concepts of  deep learning in artificial intelligence~\cite{lecun2015deep} and fairness in behavioural economics~\cite{rabin1993incorporating},  we introduce the notion of \emph{deep structured  team} (deep team for short) as an attempt to  establish a bridge between team theory and its applications.   It is to be noted that deep learning and deep team share some resemblances; for example, they both involve  multi-stage  stochastic optimization problems  over  a number of  i.i.d. random variables, and  more importantly,  the deep team problem  resembles  a deep neural network. However,  the deep team approach  is conceptually   different from deep learning  because it is  not  data-driven and its depth refers to the number of decision makers. In simple words, deep team approach  is an endeavour to provide a systematic framework to  make  the classical single-agent control algorithms  deep with respect to the number of decision makers,  which in turn  makes them applicable to  large-scale  control systems  by  exploiting   the notion of invariance principle. To this end, we borrow the notion of partially exchangeable systems,  analogous to the notion of invariance of coordinates in   physics.  A multi-agent  system is said to be partially exchangeable if the population of agents can be partitioned into a few sub-populations in such a way that  the order in which the agents  are indexed in each sub-population does not matter. 

  It is well-known that  Markov decision processes  with  partially exchangeable agents are equivalent  to  Markov decision processes coupled  through the empirical distribution of  states and actions of  agents in each sub-population~\cite{arabneydi2016new}.  Subsequently, without loss of generality, we restrict attention to  the latter formulation in this paper.   We consider  two non-classical information structures:  deep-state sharing and partial  deep-state sharing,  where  deep state  refers to  the empirical distribution  of the states of agents  in  each sub-population.
    Under deep-state sharing structure,   each decision maker has access to its local state as well as  the deep states of all sub-populations  whereas under partial deep-state sharing,   it has access to its local state and  the deep states of  only a subset (which can be empty) of sub-populations.  Under these decentralized  information structures,  we  study microscopic and macroscopic behaviours of the decision makers  and  identify globally optimal and sub-optimal fair strategies. This manuscript is a complete version of authors'  recent work in  networked control systems~\cite{JalalCDC2017,JalalACC2018,JalalACC2019} and is the generalization  of  the concept of mean-field teams introduced in~\cite{arabneydi2016new}.

In the context of game theory, mean-field games study the  non-cooperative  behaviour of a large number of exchangeable players~\cite{HuangPeter2006,weintraub2008markov,Lasry2007mean}. The solution concept  is  Nash  equilibrium and the proof   revolves around the fact that the effect of a single player on others is negligible when the number of players is sufficiently large such that it can be considered as infinite.  This   reduces    the infinite-population game  to a two-player game between a generic player and an infinite population, and since the deep state of the infinite population (i.e., mean-field)  has  deterministic dynamics,  a Nash solution  may be obtained  in terms of  the solution of two coupled  forward-backward nonlinear  partial differential equations (i.e., Fokker-Planck-Kolmogorov and  Hamilton-Jacobi-Bellman equations). The existence of such  a solution  is  established by imposing  various Lipschitz-type fixed-point conditions (that generally hold for small time horizons) or monotonicity-type assumptions (that are often difficult to verify).  For the special case of a common cost function,  mean-field games may  be used to identify an  approximate  person-by-person (Nash-bargaining) solution.  It is to be noted that  person-by-person optimality is weaker than global optimality, in general.  For example,   consider  a simple  static two-agent control problem  with the cost function $|\max(x_1,x_2)|$, $x_1,x_2 \in \mathbb{R}_{\geq 0}$.  This problem has uncountably  many  person-by-person optimal solutions (e.g. $x_1=x_2=x, \forall x \in \mathbb{R}_{\geq 0}$) but admits a unique globally optimal solution (i.e. $x_1=x_2=0$). 

Another relevant  field of research that  is closely  related to  mean-field games is  mean-field-type control problems~\cite{andersson2011maximum,bensoussan2013mean,carmona2013control}. When the cost functions and control laws of players  in this type of problem are identical, the infinite-population  social optimization problem reduces to a single-agent stochastic optimal control problem, whose dynamics and cost depend on  the distribution of  the state (also known as  McKean-Vlasov type).   A solution to this problem is not known in general,  as  it  is time-inconsistent~\cite{carmona2018probabilistic,bensoussan2013mean}.  However,  it is  possible to find  a person-by-person time-consistent solution by  formulating the mean-field type control in terms of two coupled forward-backward  equations, similar to those in mean-field games, with the  distinction that  the mean-field   is replaced by the probability density function of the state of  the generic player~\cite{carmona2018probabilistic,Saldi2018}. 

Despite the differences between cooperative mean-field games and mean-field-type control, they  both  use the simplification afforded by the infinite-population model to propose strategies that are sub-optimal, fair,  person-by-person, closed-loop in the local state  and  open-loop in the population state.  To ensure that the infinite-population solution is a reasonable approximate solution for the finite population model, the standard approach is to assume  that the solution is a continuous function.   Finding  a numerical  solution for the general case of nonlinear state dynamics,  to  the best of  our knowledge, is  still an  open problem (specially when neither fixed-point nor  monotonicity-type conditions  are  imposed on the model). 

We take a different route in this paper and  consider deep team problems that have  finite state and action spaces with non-convex cost functions and possibly multiple person-by-person solutions.  More precisely,  we  study  an arbitrary number of agents (not necessarily large) wherein  the effect of each agent on other agents is not negligible.  Note that this is a  more challenging problem compared to the typical mean-field problems where the effect of a single agent on others is normally neglected.   We   propose a team-theoretic approach to identify the  \emph{exact} globally optimal fair strategy under deep-state sharing information structure (that is closed-loop in the population state) and a globally  sub-optimal fair strategy  under partial deep-state sharing information structure (that includes open-loop strategies).  The proposed dynamic programs  are time-consistent for  any number of  agents and  their minimization is carried out over the space of local control laws. 
 In addition,  we  show that  the performance gap between the optimal and sub-optimal solutions  converges to zero  as the number of agents goes  to infinity,  without imposing any continuity assumption on the solution. Furthermore,  we propose a quantization technique   for both finite- and   infinite-population models to numerically  compute the corresponding sub-optimal solutions  without restricting to any  fixed-point or monotonicity condition.  
  It is to be noted that since the approach is neither based on the negligible effect of individual agents nor  dependent on the future trajectory of the population state,  its extension to   major-minor and common-noise problems~\cite{Nourian2013MM,cardaliaguet2015master}  introduces no  additional complication.
  
%
The rest of the paper is organized as follows.  In Section~\ref{sec:Mean-field coupled system},  the problem is formulated and   the main  contributions of this paper are outlined.  To identify an optimal solution under deep-state sharing  and  a sub-optimal one  under partial deep-state sharing, two novel dynamic programs are  proposed   in Sections~\ref{sec:main-result-problem-finite} and~\ref{sec:main_results_pmfs}, respectively.  To alleviate the computational  complexity of these  dynamic programs, their quantized counterparts are presented  in Section~\ref{sec:quantized}.  In Sections~\ref{sec:infinite-horizon} and~\ref{sec:arbitrary},  the main results are  analogously extended  to  the infinite-horizon discounted and arbitrarily asymmetric  cost functions.  The special case of major-minor deep  teams  is presented in Section~\ref{sec:MM} followed by a  numerical example  in Section~\ref{sec:numerical}. Finally, the paper is concluded in  Section~\ref{sec:conclusion}.

\section{Problem Formulation}\label{sec:Mean-field coupled system}
Throughout the paper, $\mathbb{N}$, $\mathbb{R}_{\geq 0}$, and $\mathbb{R}_{>0}$ denote the set of  natural numbers, non-negative real numbers, and positive real numbers, respectively.  $\mathbb{N}_k$ denotes the finite set of integers $\{1,\ldots,k\}$.  In addition, $\Prob{\boldsymbol \cdot}$ is the probability of a random variable; $\Exp{\boldsymbol \cdot}$ is the expectation of an event; $\ID{\boldsymbol \cdot}$ is the indicator function of a set; $\Ninf{\boldsymbol \cdot}$  denotes the infinity norm of a vector; $\xi( \boldsymbol \cdot)$ is the empirical distribution of a set; $| \boldsymbol \cdot|$ is the absolute value of a real number or the cardinality of a set, and   $\delta_{\boldsymbol \cdot}$ denotes the Dirac measure with a unit mass concentrated at one point.  The short-hand notation $x_{1:t}$ is used  to denote vector $(x_1,\ldots,x_t)$. For any pair of integers $i,j \leq n \in \mathbb{N}$, $\sigma_{i,j}(\mathbf x)$ denotes the permuted version of vector $\mathbf x=(x_1,\ldots,x_n)$  such that the $i$-th element of $\sigma_{i,j}(\mathbf x)$ is $x_j$,  the $j$-the element of $\sigma_{i,j}(\mathbf x)$ is $x_i$, and the other elements are the same as of  $\mathbf x$.  The abbreviation \emph{a.s.}  stands for almost surely. For any $n \in \mathbb{N}$,   $\Binopdf(\boldsymbol \cdot,n,p)$ denotes the binomial probability distribution of $n$ trials with success probability~$p \in [0,1]$. 
Furthermore, for any $n,\qss \in \mathbb{N}$ and any  finite set $\mathcal{X}$, different  spaces are defined as described in Table~\ref{table:space}.
\begin{table}[t!]
\caption{ }\label{table:space}
\centering
\setlength\tabcolsep{0em}
\begin{tabular}{l} 
  \toprule
\multicolumn{1}{c}{\textbf{Space of probability measures over $\mathcal{X}$}}\\
\midrule
$\mathcal{P}(\mathcal{X})=\{(\alpha_1,\ldots,\alpha_{|\mathcal{X}|}) \big| \alpha_i \in [0,1], i \in \mathbb{N}_{|\mathcal{X}|}, \sum_{i=1}^{|\mathcal{X}|} \alpha_i=1\}$\\
\midrule
\multicolumn{1}{c}{\textbf{Space of empirical distributions  over $\mathcal{X}$ with $n$ samples} }\\
\midrule 
$\MeanSpace_n(\mathcal{X})=\{(\alpha_1,\ldots,\alpha_{|\mathcal{X}|}) \big| \alpha_i \in \{0, \frac{1}{n},\ldots,1\}, i \in \mathbb{N}_{|\mathcal{X}|}, \sum_{i=1}^{|\mathcal{X}|} \alpha_i=1\}$\\
\midrule
\multicolumn{1}{c}{\textbf{ $|\mathcal{X}|$-fold unit interval} }\\
\midrule
$\mathcal{I}(\mathcal{X})=\{(\alpha_1,\ldots,\alpha_{|\mathcal{X}|}) \big| \alpha_i \in [0,1], i \in \mathbb{N}_{|\mathcal{X}|}\}$\\
\midrule
\multicolumn{1}{c}{\textbf{$\mathcal{I}(\mathcal{X})$ uniformly quantized with$\qss$ levels  } }\\
\midrule
$\mathcal{Q}_\qss (\mathcal{X})=\{(\alpha_1,\ldots,\alpha_{|\mathcal{X}|}) \big| \alpha_i \in \{0, \frac{1}{\qss},\ldots,1\}, i \in \mathbb{N}_{|\mathcal{X}|}\}$\\
  \bottomrule
\end{tabular}
\end{table}
%
 Let $x_{1:n}$ denote a vector of $n \in \mathbb{N}$  samples from set $\mathcal{X}:=\{a_1,\ldots,a_{|\mathcal{X}|}\}$, where $x_i \in \mathcal{X}, i \in \mathbb{N}_n $. The empirical distribution function $\xi: \prod_{i=1}^n \mathcal{X} \rightarrow \MeanSpace_n(\mathcal{X})$  is defined as  a real-valued vector of size $|\mathcal{X}|$ such that
$\xi(x_{1:n})(a_j)= \frac{1}{n}\sum_{i=1}^n \mathds{1}(x_i=a_j), j \in \mathbb{N}_{|\mathcal{X}|}.
$
\subsection{Model}
There are various  applications in which  the population of agents can be  partitioned into a few sub-populations in such a way  that the order of  indexing  of agents in each sub-population  is not important.  For example, in a smart grid,  demands in a region  may be classified  as residential and commercial,  and the numbering of the demands in each class does not affect  the aggregate  consumed energy. Similarly, in swarm robotics,   robots may be categorized  into a few  groups  with identical characteristics such as leaders and followers  where neither  the  dynamics of motion nor the status of the swarm  depends on the way the robots are indexed in each group.   Such applications  may be modelled as described next.

Consider a discrete-time control system consisting of a finite population of  agents, where agents are partitioned into $K \in \mathbb{N}$ disjoint sub-populations.   Denote by $\mathcal{K}$ the set of sub-populations,  by $\mathcal{N}^k$  the agents of sub-population $k \in \mathcal{K}$, and by  $\mathcal{N}$ the entire population of agents; note that $\mathcal{N}=\cup_{k \in \mathcal{K}} \mathcal{N}^k$.  Given the control horizon $T \in \mathbb{N}$, let the state, action, and the noise of agent $i \in \mathcal{N}^k$ of sub-population $k \in \mathcal{K}$ at time $t \in \mathbb{N}_T$  be denoted by  $x^i_t \in \mathcal{X}^k$, $u^i_t \in \mathcal{U}^k$, and $w^i_t \in \mathcal{W}^k$, respectively. The spaces $\mathcal{X}^k$, $\mathcal{U}^k$ and $\mathcal{W}^k$ of every sub-population $k \in \mathcal{K}$ are finite and  do not depend on the size of sub-population $k \in \mathcal{K}$, i.e. $|\mathcal{N}^k|$.   For the entire population,  the joint state, joint action, and joint noise are analogously denoted by $\mathbf x_t=(x^i_t)_{i \in \mathcal{N}} \in \mathcal{X}$, $\mathbf u_t=(u^i_t)_{i \in \mathcal{N}} \in \mathcal{U}$,  and $\mathbf w_t=(w^i_t)_{i \in \mathcal{N}} \in \mathcal{W}$  at time $t \in \mathbb{N}_T$.  

 Let $\mathfrak{D}^k_t$ denote the  empirical distribution of the states and actions  of sub-population $k \in \mathcal{K}$ at time $t \in \mathbb{N}_T$, i.e.,   
\begin{equation}\label{eq:joint-mean-field-k}
\mathfrak{D}^k_t=\Emp{(x^i_t,u^i_t)_{i \in \mathcal{N}^k}} \in \MeanSpace_{|\mathcal{N}^k|}(\mathcal{X}^k \times \mathcal{U}^k).
\end{equation}
Similarly, let $d^k_t$ denote the empirical distribution of the states of sub-population $k \in \mathcal{K}$ at time $t \in \mathbb{N}_T$, i.e.,  
\begin{equation}\label{eq:mean-field-k}
d^k_t=\Emp{(x^i_t)_{i \in \mathcal{N}^k}} \in \MeanSpace_{|\mathcal{N}^k|}(\mathcal{X}^k).
\end{equation}
Define 
$\boldsymbol{\mathfrak{D}}_t:=(\mathfrak{D}_t^1,\ldots,\mathfrak{D}^K_t)$ and $  \mathbf d_t:=(d_t^1,\ldots,d^K_t)$. 
For ease of reference,  the empirical distribution of states  is called  \textit{deep state}  in the sequel. Denote  by $\MeanSpace$ the space of  realizations $\mathbf d_t$, i.e., $\MeanSpace:=\prod_{k \in \mathcal{K}} \MeanSpace_{|\mathcal{N}^k|}(\mathcal{X}^k)$, $\NT$. 

For any  $k \in \mathcal{K}$, the initial state of agent $i$ of sub-population~$k$ is denoted by  $x^i_1 \in \mathcal{X}^k$, and at time $t \in \mathbb{N}_T$ its state evolves as follows:
\begin{equation}\label{eq:dynamics-f-mean-field}
x^i_{t+1}=f^k_t(x^i_t,u^i_t, \boldsymbol{\mathfrak{D}}_t, w^i_t),
\end{equation}
 where  the primitive random variables $\{\mathbf x_1, \mathbf w_1,\ldots,\mathbf w_T\}$ are defined on a common probability space and are mutually independent. The  dynamics~\eqref{eq:dynamics-f-mean-field} can be \textrm{equivalently} represented  in terms of the transition probability matrix  such that:
\begin{multline}\label{eq:relation-models}
\mathbb{P}^k(x^i_{t+1} \mid  x^i_t,u^i_t, \boldsymbol{\mathfrak{D}}_t): =\\
  \sum_{w^i \in \mathcal{W}^k} \ID{x^i_{t+1}=f^k_t(x^i_t,u^i_t, \boldsymbol{\mathfrak{D}}_t, w^i_t)} \Prob{w^i_t=w^i}.
\end{multline}
In the sequel, we occasionally  interchange the two equivalent representations \eqref{eq:dynamics-f-mean-field} and \eqref{eq:relation-models}, for  ease of display.  Let the per-step cost function be  uniformly bounded and denoted by $c_t(\boldsymbol{\mathfrak{D}}_t) \in \mathbb{R}_{\geq 0}$  at time $t$.  Note that the social cost function is a special case of the above cost function. To see this, let $c^k_t(x^i,u^i_t,\boldsymbol{\mathfrak{D}}_t)$ denote the per-step cost of agent $i \in \mathcal{N}^k$ of sub-population $k \in \mathcal{K}$. Then,  one arrives at:
\begin{equation}
c_t(\boldsymbol{\mathfrak{D}}_t):=\sum_{k \in \mathcal{K}} \frac{1}{|\mathcal{N}^k|} \sum_{i \in \mathcal{N}^k} c^k_t(x^i,u^i_t,\boldsymbol{\mathfrak{D}}_t).
\end{equation}

 In this paper, we make the following three  assumptions on the primitive random variables.
\begin{Assumption}\label{assumption:exchangeable}
For  any sub-population $k \in \mathcal{K}$, the primitive  random variables $(w^i_t)_{i \in \mathcal{N}^k}$, $t \in \mathbb{N}_T$,   are exchangeable.
\end{Assumption}
\begin{Assumption}\label{assumption:iid}
The primitive random variables  $(w^i_t)_{i \in \mathcal{N}}$, $t \in \mathbb{N}_T$, are mutually independent across agents. In addition,  for  each sub-population $k \in \mathcal{K}$, random variables  $(w^i_t)_{i \in \mathcal{N}^k}$   are identically distributed with probability mass functions~$\mathbb{P}_{W^k_t}$.
\end{Assumption}
\begin{Remark}
Note that Assumptions~\ref{assumption:exchangeable} and~\ref{assumption:iid} do not impose any restriction on the  probability distribution of  initial states. 
\end{Remark}
\begin{Assumption}\label{assumption:iid-x}
The primitive random variables    $(x^i_1)_{i \in \mathcal{N}}$  are mutually independent across agents, and  for  each sub-population $k \in \mathcal{K}$,  random variables $(x^i_1)_{i \in \mathcal{N}^k}$  are identically distributed with probability mass functions $\mathbb{P}_{X^k_1}$.
\end{Assumption}

Denote by  $I^i_t  \subseteq \{\mathbf x_{1:t}, \mathbf u_{1:t-1} \}$ the information  set of  agent $i \in \mathcal{N}$ at time $t \in \mathbb{N}_T$, i.e.,
\begin{equation}\label{eq:control-law-general}
u^i_t=g^i_t(I^i_t),
\end{equation}
where  function $g^i_t$  is called the  control law of agent $i$ at time $t \in \mathbb{N}_T$. The set of control laws $\mathbf{g}:=\{(g^i_t)_{i \in \mathcal{N}}\}_{t=1}^T$ is defined as the \emph{strategy} of  the system.

\begin{Definition}[\textbf{Partially exchangeable (fair) strategies}]\label{def:fair-strategies}
A strategy $\mathbf g$ is said to be partially exchangeable  if for any pair of agents $(i,j) \in \mathcal{N}^k$ of  any sub-population $k \in \mathcal{K}$ at any time $\NT$,
$\sigma_{i,j} \left(g^s_ t(I^s_t)\right)_{s \in \mathcal{N}} = \left(g^s_ t( \sigma_{i,j}  I^s_t)\right)_{s \in \mathcal{N}},$
i.e., exchanging  agents $i$ and $j$ has no effect on the strategy.
\end{Definition}

\subsection{Admissible strategy}\label{sec:admissible}

 The set of admissible strategies is  defined as  the set of partially exchangeable (fair) strategies, where no agent is privileged or discriminated by the way it is indexed in a sub-population. The restriction to fair strategies may be viewed from two angles: a constraint that must be satisfied or an assumption that is limiting.   In this paper, we focus on the former case and present applications in which such a restriction is desirable and practical. It is substantiated by numerous experimental data  in behavioural economics that humans make their decisions based not only  on rationality but also on whether or not the decisions are fair~\cite{rabin1993incorporating,falk2006theory}. For example, an unfair resource allocation in a  grid  can  lead to protest and anarchy  among  users, due to the discrimination against some users,  even if such a strategy yields the lowest possible social cost function. This means that a Pareto-optimal \emph{unfair} strategy is  not necessarily  a sustainable equilibrium.  To  learn more about the  importance of fairness, the interested reader is referred to a pivotal counterexample in behavioural economics called the  ultimatum game~\cite{Guth1982} for  a seemingly irrational behaviour.  In addition, fair strategies are important in  control theory   because they provide robustness,  where, for instance,  it is   desirable  to distribute the total load of a network in a fairly manner  among  servers in order to increase  the  life-time of the servers   as well as the  robustness of the network (in  case   a server fails).  Furthermore, since finding a Pareto-optimal solution is  computationally  expensive, as described  in Section~\ref{sec:introduction},  the optimal fair strategy is of particular interest in practice, as  a  simpler (yet more tractable) alternative.

Two  decentralized information structures are investigated in this paper.  The first one is referred to as  the \emph{deep-state sharing} (DSS), where for any~$i \in \mathcal{N}$,  agent~$i$ at time $\NT$ observes  its local state $x^i_t$   as well as the history of the deep states  of all sub-populations, i.e.,
\begin{equation}
I^i_t=\{x^i_{t},\mathbf{d}_{1:t}\}. \tag{DSS}
\end{equation}
In practice, there are different ways to share the deep state among agents. For example, in cellular communications,  the deep state  may be collected and transmitted  to all agents by the base station, while  in  swarm robotics  the deep state may be computed in a distributed manner using consensus-based algorithms.  The second information structure is more general than DSS, and is   called the \emph{partial deep-state sharing} (PDSS),  where for any $i \in \mathcal{N}$,  agent $i \in \mathcal{N}$ at time $\NT$  observes  its local state $x^i_t$ and  the history of the deep states of  a subset of sub-populations $\mathcal{S} \subseteq \mathcal{K}$, i.e.,
\begin{equation}
I^i_t= \{x^i_{t},(d^k_{1:t})_{k \in \mathcal{S}}\}. \tag{PDSS}
\end{equation}
Note that if $\mathcal{S}=\mathcal{K}$, PDSS is the same as DSS and if $\mathcal{S}=\emptyset$, PDSS is fully decentralized. 
Note also that  DSS and PDSS respect the privacy of agents  by not sharing the local state of each agent.
The performance of any strategy $\mathbf g$ is  described by:
\begin{equation}\label{eq:J-general-def}
J_N(\mathbf{g})= \mathbb{E}^{\mathbf{g}} \big[ \sum_{t=1}^T c_t(\boldsymbol{\mathfrak{D}}_t) \big],
\end{equation}
where  the subscript $N$  denotes  the dependence  of the cost function  to the number of agents, and the expectation is taken with respect to the probability measures  induced by $\mathbf g$. 
\begin{Problem}\label{prob:MFS-finite}
 For deep-state sharing information structure, find the optimal fair strategy $\mathbf{g}^\ast$  such that for every fair strategy~$\mathbf{g}$,
$J_N(\mathbf g^\ast) \leq J_N(\mathbf g)$.
\end{Problem}
Denote by~$\mathcal{S}^c$ the complement set of $\mathcal{S}$, i.e., $\mathcal{S}^c=\mathcal{K} \backslash \mathcal{S}$, and let $n$ be the size of the smallest sub-population whose deep state is not observed, i.e., $n:= \min_{k \in \mathcal{S}^c} |\mathcal{N}^k|$. 
\begin{Problem}\label{prob:PMFS-finite}
For partial deep-state sharing information structure,  find an $\varepsilon(n)-$optimal  fair strategy $\mathbf{g}$ such  that
$J_N(\mathbf {g}) \leq J_N(\mathbf g^\ast)+ \varepsilon(n)$,
where $\varepsilon(n) \in \mathbb{R}_{>0}$ and $\lim_{n \rightarrow \infty} \varepsilon(n)=0$.
\end{Problem}
\begin{Remark}
\emph{For  an  exchangeable system, Pareto-optimal and  optimal  fair strategies  are not necessarily the same; however,   for a linear quadratic model,  they are identical under  DSS~\cite{arabneydi2016new}. 
}
\end{Remark}

Since both DSS and PDSS are non-classical information structures, the computational complexity of finding a solution to Problems~\ref{prob:MFS-finite} and \ref{prob:PMFS-finite} is NEXP,  in general~\cite{Bernstein2002complexity}. The main contributions of the present paper are spelled out below.
\begin{enumerate}
\item We develop a dynamic program that can be used to  find an optimal solution for Problem~\ref{prob:MFS-finite} (Theorem~\ref{thm:mfs-finite}). We also  show that the computational complexity of finding the solution in both space and time is polynomial  with respect to the number of agents in each sub-population and is linear  with respect to the control horizon (Corollary~\ref{remark:computational-space-time}).

\item Although polynomial complexity  is less than than exponential,   it could  still be high when the size of population is  medium or large.
For medium populations, we propose Theorems~\ref{thm:mfs-finite-independent}  and~\ref{thm:mfs-quantized} to alleviate the computational complexity (in time and space).  In particular,  show that the dynamics of the deep state in Theorem~\ref{thm:mfs-finite-independent}  is described by an equation that we refer to as  the \emph{Deep Chapman-Kolmogorov}  (DCK) equation, with a  structure analogous to that in convolutional  neural networks (Subsection~\ref{remark:chapman}). 

\item    For a large population, we consider Problem~\ref{prob:PMFS-finite}   because when  some sub-populations are large, it may not be feasible to collect and share their deep states among agents. We show that such information sharing has a negligible effect on the optimal performance of the system. In particular, we  develop a dynamic programming decomposition for Problem~\ref{prob:PMFS-finite}  that  provides  an $\varepsilon(n)$-optimal  strategy, where  $\varepsilon(n)$ converges to zero at the rate $1/\sqrt n$ (Theorem~\ref{thm:pmfs-finite}).  This dynamic program does not depend on the size of  sub-populations~$ \mathcal{S}^c$.

\item   For the numerical computation  of the dynamic program of Theorem~\ref{thm:pmfs-finite},  it is required, in general, to solve a non-smooth non-convex optimization problem.
 We propose a quantized solution  and prove that the quantization error converges to zero at a rate  inversely proportional to the number of quantization levels (Theorem~\ref{thm:pmfs-quantized}). An immediate consequence  of this result is that if the number of quantization levels is greater than $\sqrt{n}$, then the quantized solution will converge to the optimal solution  at the same rate that the unquantized solution does (Corollary~\ref{cor:rate}).

\item  We extend our main results to infinite-horizon discounted cost  (Theorems~\ref{thm:mfs-finite-discounted} and~\ref{thm:pmfs-finite-discounted}).  It is shown that DSS  strategy is stationary with respect to the observed deep states whereas PDSS strategy  is not (Remark~\ref{remark:infinite}).

\end{enumerate}

We define a number of short-hand notations to ease the exposition of the results and proofs in the  sequel.   Let   $ \IX:=  \prod_{k \in \mathcal{K}} \mathcal{I}(\mathcal{X}^k)$,  and  given  any subset $\mathcal{R} \subseteq \mathcal{K}$  and any scalar $\qss \in \mathbb{N}$,  define the following spaces: 
\begin{align}\label{eq:def-spaces-short}
\begin{cases}
\mathcal{P}_{\scriptscriptstyle{\mathcal{R}}}&:=\prod_{k \in \mathcal{R}} \mathcal{E}_{|\mathcal{N}^k|}(\mathcal{X}^k) \times  \prod_{k \in \mathcal{R}^c} \mathcal{P}(\mathcal{X}^k),\\
 \QR&:=  \prod_{k\in \mathcal{R}} \mathcal{Q}_{\qss}(\mathcal{X}^k) \times \prod_{k \in \mathcal{R}^c} \mathcal{E}_{|\mathcal{N}^k|}(\mathcal{X}^k), \\ 
 \IXR&:= \prod_{k\in \mathcal{R}} \mathcal{I}(\mathcal{X}^k)\times \prod_{k \in \mathcal{R}^c} \mathcal{E}_{|\mathcal{N}^k|}(\mathcal{X}^k),
 \end{cases}
\end{align} 
where $\MeanSpace \subset  \mathcal{P}_{\scriptscriptstyle{\mathcal{R}}} \subset \IX \text{ and }   \QR \subset \IXR \subset \IX$. Denote by $Q:\IXR \rightarrow \QR$ the quantizer function that maps  every point $\mathbf z \in \IXR$ to its nearest point $\mathbf q \in \QR$, i.e., 
$
Q(\mathbf z) \in \argmin_{\mathbf q \in \QR} \Ninf{\mathbf z - \mathbf q },
$
which implies that: 
$\Ninf{\mathbf z - Q(\mathbf z)} \leq \frac{1}{2\qss}, \forall \mathbf z \in \IXR$.

Let $\gamma^k:\mathcal{X}^k \rightarrow \mathcal{U}^k$ be the mapping from the local state space  to the local action space of sub-population $k \in \mathcal{K}$, and  $\boldsymbol \gamma:=\{\gamma^1,\ldots,\gamma^K\} \in \mathcal{G}$,  where $\mathcal{G}$  denotes the space  of all mappings~$\boldsymbol \gamma$.  Let also $\MW$ and $\underline{\mathbf w}_t$ denote the empirical distribution of the  local noises of sub-population $k \in \mathcal{K}$ and the entire system, respectively,  at time $t \in \mathbb{N}_T$, i.e.,  
\begin{align}\label{eq:mean-field-k-noise}
\begin{cases}
\MW&:=\Emp{(w^i_t)_{i \in \mathcal{N}^k}} \in \mathcal{E}_{|\mathcal{N}^k|}(\mathcal{W}^k),  \\
\underline{\mathbf w}_t&:= (\underline{w}_t^1,\ldots,\underline{w}^K_t)\in \underline{\mathcal{W}}:= \prod_{k \in \mathcal{K}}  \mathcal{E}_{|\mathcal{N}^k|}(\mathcal{W}^k).
\end{cases}
\end{align}
Then,  define the following functions at  every time $\NT$ and  for  any $\mathbf z=\{z^1,\ldots,z^K\} \in \IX$, $\boldsymbol \gamma \in \mathcal{G}$,  $\underline{\mathbf w}_t\in \underline{\mathcal{W}}$, and $k \in \mathcal{K}$:
 
1)  For any $x \in \mathcal{X}^k$ and  $u \in \mathcal{U}^k$,  $\phi^k(\mathbf z,\boldsymbol \gamma)(x,u) \in [0,1]$ is defined as 
$z^k(x) \ID{u=\gamma^k(x)}$,
where  its augmented form is:
\begin{equation}  \label{eq:phi-def-function}
  \phi(\mathbf z, \boldsymbol \gamma):=(\phi^1(\mathbf z, \boldsymbol \gamma),\ldots,\phi^K(\mathbf z, \boldsymbol \gamma)).
\end{equation}

2) For any   $ y \in \mathcal{X}^k$,  $\bar{f}^k_t(\mathbf z,\boldsymbol \gamma, \MW)(y) \in [0,1]$ is  defined  as
\begin{equation}\label{eq:def-bar-f-k}
\sum_{w \in \mathcal{W}^k} \sum_{x \in \mathcal{X}^k} z^k(x)   
 \ID{f^k_t(x,\gamma^k(x), \phi(\mathbf z,\boldsymbol \gamma), w)= y } \MW(w), 
\end{equation}
where  its augmented form is:
\begin{equation}\label{eq:bar-f-augmented-def}
\bar f_t(\mathbf z,\boldsymbol \gamma, \underline{\mathbf w}_t):=(\bar f^1_t(\mathbf z,\boldsymbol \gamma,\underline{w}^1_t),\ldots,\bar f^K_t(\mathbf z,\boldsymbol \gamma,\underline{w}^K_t)).
\end{equation}

3) For any $ y \in \mathcal{X}^k$,   $\FS^k_t(\mathbf z, \boldsymbol \gamma)(y) \in [0,1]$ is defined~as:
\begin{equation}\label{eq:hat-f-def}
\sum_{x \in \mathcal{X}^k} z^k(x) \ProbK{y \mid x, \gamma^k(x), \phi(\mathbf z, \boldsymbol \gamma)}.
\end{equation}

4)  For any per-step cost $c_t$,  define the following  nonnegative real function: 
\begin{equation}\label{eq:hat-l-def}
\ell_t(\mathbf z, \boldsymbol \gamma):=c_t(\phi(\mathbf z,\boldsymbol \gamma)). 
\end{equation}

\section{Main results for Problem~\ref{prob:MFS-finite}}\label{sec:main-result-problem-finite}
To find a solution to  Problem~1, Witsenhausen's standard form could be used  to develop a dynamic programming decomposition~\cite{Witsenhausen1973}.  However,  the resultant dynamic program  would be  intractable, and since the size of its information state increases with time, it could not be extended to  infinite horizon. Alternatively, one can  use  the  common information approach~\cite{Nayyar2013CIA} to construct a dynamic program  in terms of the conditional probability of the joint state, given the common information, i.e., $\Prob{\mathbf x_t \mid \mathbf d_{1:t}}$. In such a case, it is shown in~\cite{JalalCDC2014} by forward induction that 
if the initial states as well as noise processes are exchangeable,  $\Prob{\mathbf x_t \mid \mathbf d_{1:t}}$ is  also exchangeable under  fair strategies, and  hence can be represented   by~$\mathbf d_t$.   For  the general case of non-exchangeable initial states,  however,  the result of~\cite{JalalCDC2014} does not hold as $\Prob{\mathbf x_t \mid \mathbf d_{1:t}}$ is not  necessarily exchangeable. Thus, we  present a direct approach to obtain  a dynamic programming decomposition in terms of $\mathbf d_t$,  regardless of the probability distribution of  initial states. A salient feature   of this  method is  to  identify the structure of   the transition probability matrix of deep state $\mathbf d_t$, which proves to be useful  not only for the numerical computations  but also  for the convergence analysis of Problem~\ref{prob:PMFS-finite}.

\begin{Lemma}\label{lemma:fair}
When  attention is focused on fair strategies,   the control laws of  DSS and PDSS strategies are identical in each sub-population, i.e., $
g^i_t=g^j_t=:g^k_t, i,j \in \mathcal{N}^k, k \in \mathcal{K}$.
\end{Lemma}
\begin{proof}
The proof follows directly  from equation \eqref{eq:control-law-general}, the definitions of DSS and PDSS, and  Definition~\ref{def:fair-strategies}.
\end{proof}
According to equation~\eqref{eq:control-law-general} and  Lemma~\ref{lemma:fair}, for any $k \in \mathcal{K}$ and $i \in \mathcal{N}^k$, the control law of agent $i$ of sub-population $k$ at time $\NT$ under DSS is  $g^k_t: \MeanSpace^t \times \mathcal{X}^k \rightarrow \mathcal{U}^k$, i.e., 
\begin{equation}\label{eq:mfs-is-fair}
u^i_{t}=g^k_t(x^i_t,\mathbf d_{1:t}).
\end{equation}
Using  the information decomposition proposed in~\cite{Nayyar2013CIA}, we split $g^k_t$  into two parts for any $k \in \mathcal{K}$ and $\NT$. More precisely,    define function $\psi^k_t:\MeanSpace^t \rightarrow \mathcal{G}$ as follows:
\begin{equation} \label{eq:psi-g}
\psi^k_t(\mathbf d_{1:t}):=g^k_t(\boldsymbol \cdot,\mathbf{d}_{1:t}).
\end{equation}
Then, from \eqref{eq:mfs-is-fair} and \eqref{eq:psi-g},  one has:
\begin{equation}\label{eq:gamma-u}
u^i_t=\gamma^k_t(x^i_t),
\end{equation}
where  $\gamma^k_t:\mathcal{X}^k \rightarrow \mathcal{U}^k$ is defined by $\psi^k_t$, i.e.,
\begin{equation} \label{eq:psi-gamma}
\gamma^k_t:=\psi^k_t(\mathbf d_{1:t}).
\end{equation}
According to~\eqref{eq:gamma-u} and~\eqref{eq:psi-gamma}, the action of  agent $i$ of sub-population $k$ at time~$t$ is determined by two functions $\psi^k_t$ and $\gamma^k_t$.  In the sequel, we refer to the functions $\boldsymbol \gamma_t:=\{\gamma^1_t,\ldots,\gamma^K_t\} \in \mathcal{G}$  as  the \emph{local laws} and to $\boldsymbol \psi_t:=\{\psi^1_t,\ldots,\psi^K_t\} $ as the \emph{global laws}.  From  \eqref{eq:joint-mean-field-k}, \eqref{eq:mean-field-k} and \eqref{eq:gamma-u},  it follows that for any  $k \in \mathcal{K}$, $\NT$, $x \in \mathcal{X}^k$ and  $u \in \mathcal{U}^k$:
\begin{equation}\label{eq:joint mean-field-proof-1}
\mathfrak{D}^k_{t}(x,u) =\frac{1}{|\mathcal{N}^k|} \sum_{i \in \mathcal{N}^k} \ID{x^i_t=x}\ID{\gamma^k_t(x^i_t)=u}.
\end{equation}
Subsequently, it results from~\eqref{eq:phi-def-function} and~\eqref{eq:joint mean-field-proof-1} that:
\begin{equation}\label{eq:proof-joint-mean-field-1}
\boldsymbol{\mathfrak{D}}_{t}=\phi(\mathbf d_t, \boldsymbol \gamma_t).
\end{equation}
We  show that  the deep state evolves in a Markovian manner with respect to the local laws, which  means that the history of the deep state  except the most recent one  can be ignored.
\begin{Theorem}\label{thm:finite-mfs-polynomial}
Let Assumption~\ref{assumption:exchangeable} hold. For any $k \in \mathcal{K}$ and $\NT$, the dynamics of the deep state of sub-population $k$ at time~$t$ can be expressed by: 
\begin{equation}\label{eq:def-bar-f-k-thm}
d^k_{t+1}\substack{{a.s.}\\=} \bar f^k_t(\mathbf d_t,\boldsymbol \gamma_t,\MW),
\end{equation}
where $\MW$ and $\bar f^k_t$ are defined by~\eqref{eq:mean-field-k-noise} and~\eqref{eq:def-bar-f-k}, respectively. In addition, the deep state of the entire population evolves as:
\begin{equation}\label{eq:def-bar-f-thm}
\mathbf d_{t+1}\substack{{a.s.}\\=} \bar f_t(\mathbf d_t,\boldsymbol \gamma_t,\underline{\mathbf w}_t),
\end{equation}
where  $\underline{\mathbf w_t}$ and $\bar f_t$ are given by~\eqref{eq:mean-field-k-noise} and~\eqref{eq:bar-f-augmented-def}, respectively.
\end{Theorem}
\begin{proof}
According to \eqref{eq:mean-field-k}, \eqref{eq:dynamics-f-mean-field}, \eqref{eq:gamma-u} and \eqref{eq:proof-joint-mean-field-1}, for any $y \in \mathcal{X}^k$, one has:
\begin{equation}\label{eq:mfs-proof-poly-1}
d^k_{t+1}(y)= \frac{1}{|\mathcal{N}^k|} \sum_{i \in \mathcal{N}^k} \ID{f^k_t(x^i_t,\gamma^k_t(x^i_t),\boldsymbol{\mathfrak{D}}_t,w^i_t)=y}.
\end{equation}
For any $j \in \mathbb{N}_{|\mathcal{N}^k|},$  let $\sigma_j$ shift the vector $\mathbb{N}_{|\mathcal{N}^k|}$ circularly by $j$ positions, i.e.,  $\sigma_j(i) :=i+j$ if $ i+j \leq |\mathcal{N}^k|$; otherwise, $\sigma_j(i) :=i+j -|\mathcal{N}^k|$. From  Assumption~\ref{assumption:exchangeable}, it follows that:
\begin{equation}\label{eq:proof-iid-nm}
d^k_{t+1}(y)\substack{{a.s.}\\{=}} \frac{1}{|\mathcal{N}^k|}\SUMI \ID{f^k_t(x^i_t,\gamma^k_t(x^i_t),\boldsymbol{\mathfrak{D}}_t,w^{\sigma_j(i)}_t)=y}.
\end{equation}
It is now possible to  compute $|\mathcal{N}^k|  d^k_{t+1}(y) $ by summing up \eqref{eq:proof-iid-nm} over all  $j \in \mathbb{N}_{|\mathcal{N}^k|}$,  which is almost surely equal to:
\begin{align}\label{eq:eq:proof-iid-nm-1}
& \frac{1}{|\mathcal{N}^k|} \SUMI \sum_{j \in \mathcal{N}^k} \ID{f^k_t(x^i_t,\gamma^k_t(x^i_t),\boldsymbol{\mathfrak{D}}_t,w^{\sigma_j(i)}_t)=y}  \nonumber\\
&=\frac{1}{|\mathcal{N}^k|}   \sum_{w \in \mathcal{W}^k} \sum_{x \in \mathcal{X}^k}  \SUMI  \sum_{j \in \mathcal{N}^k} \ID{x^i_t=x} \ID{w^{\sigma_j(i)}_t=w}    \nonumber \\
& \qquad \times \ID{f^k_t(x,\gamma^k_t(x),\boldsymbol{\mathfrak{D}}_t,w)=y} \nonumber\\
&=   \sum_{w \in \mathcal{W}^k} \sum_{x \in \mathcal{X}^k}  \SUMI  \ID{x^i_t=x} \ID{f^k_t(x,\gamma^k_t(x),\boldsymbol{\mathfrak{D}}_t,w)=y}  \nonumber \\
&  \qquad \times \frac{1}{|\mathcal{N}^k|}  \sum_{j \in \mathcal{N}^k} \ID{w^{\sigma_j(i)}_t=w}      \nonumber\\
& \substack{(a)\\=}   \hspace{-.2cm}  \sum_{w \in \mathcal{W}^k}  \hspace{-.05cm} \sum_{x \in \mathcal{X}^k}   \hspace{-.05cm}  \SUMI   \hspace{-.1cm}  \ID{x^i_t \hspace{-.1cm}=\hspace{-.1cm}x} \ID{f^k_t(x,\gamma^k_t(x),\boldsymbol{\mathfrak{D}}_t,w)\hspace{-.1cm}=\hspace{-.1cm}y} \MW(w)     \nonumber\\
&\substack{(b)\\=}   |\mathcal{N}^k| \hspace{-.15cm} \sum_{w \in \mathcal{W}^k} \sum_{x \in \mathcal{X}^k}  d^k_t(x)  \ID{f^k_t(x,\gamma^k_t(x),\boldsymbol{\mathfrak{D}}_t,w)=y} \MW(w),
\end{align}
where $(a)$ and $(b)$ follow from \eqref{eq:mean-field-k} and~\eqref{eq:mean-field-k-noise} , respectively. Equation~\eqref{eq:def-bar-f-k-thm} follows now from equations~\eqref{eq:def-bar-f-k},~\eqref{eq:joint mean-field-proof-1} and the above equation. In addition, equation~\eqref{eq:def-bar-f-thm} follows from~\eqref{eq:mean-field-k-noise},~\eqref{eq:bar-f-augmented-def} and~\eqref{eq:def-bar-f-k-thm}.
\end{proof} 
\begin{Remark}\label{remark:control_law_independence}
\emph{It is to be noted  that the result of Theorem~\ref{thm:finite-mfs-polynomial} holds irrespective of the global laws $\boldsymbol{\psi}_{1:t}$. }
\end{Remark}
From Theorem~\ref{thm:finite-mfs-polynomial},  the transition probability matrix of the deep state of entire population can be presented as follows:
\begin{equation}\label{eq:mftpm}
\Compress
\Prob{\mathbf d_{t+1} \mid \mathbf d_{t},\boldsymbol \gamma_{t}}=
\sum_{\underline{\mathbf w}  \in \underline{\mathcal{W}}}  
\ID{\mathbf d_{t+1}=\bar f_t(\mathbf d_t,\boldsymbol \gamma_t,\underline{\mathbf w})}
\Prob{\underline{\mathbf w}_t=\underline{\mathbf w}}.
\end{equation}

\begin{Remark}
\emph{The deep-state process $\mathbf d_{1:T}$ is not necessarily a controlled Markov process under joint actions $\mathbf u_{1:T-1}$, i.e., 
\begin{equation}
\Prob{\mathbf d_{t+1} \mid \mathbf d_{1:t}, \mathbf u_{1:t}} \neq \Prob{\mathbf d_{t+1} \mid \mathbf d_{t}, \mathbf u_{t}}.
\end{equation}
As a counterexample, consider a homogeneous population of agents, where $x^i_{t+1}=( 1- x^i_t)u^i_t + x^i_t(1-u^i_t)$, $x^i_t,u^i_t\in \{0,1\}$, $i \in \mathbb{N}\backslash\{1,2\}, \NT,$ with  known initial states $\mathbf x_1$.  It is not possible to express  $d_{t+1}$  as a function of  $d_{t}$ and $\mathbf u_t$, but it is feasible to write it as a function of  $\mathbf u_{1:t}$ (by simply computing $\mathbf x_{1:t}$ due to the deterministic dynamics). Consequently,  the Markov property does not hold. 
 }
\end{Remark}
  Define value functions $\{\VM_1,\ldots,\VM_T,\VM_{T+1}\}$ backward in time such that 
$\VM_{T+1}(\mathbf{d})~=~0$, $\forall \mathbf d \in \MeanSpace.$
Also, for  any $\NT$ and  $\mathbf d \in \MeanSpace$, define:
\begin{equation}\label{eq:thm-dp-mfs-finie-t}
\VM_t(\mathbf{d})=\min_{\boldsymbol \gamma \in \mathcal{G}} \left(\LS_t\left(\mathbf d,\boldsymbol \gamma \right) +\Exp{\VM_{t+1} \left(\bar f_t(\mathbf d,\boldsymbol \gamma,\underline{\mathbf w}_t) \right)}\right).
\end{equation}
Denote by   $\boldsymbol \psi^{*}_t(\mathbf{d})=\{\psi^{*,1}_t(\mathbf{d}),\ldots,\psi^{*,K}_t(\mathbf{d})\}$  an argmin of the right-hand side of the above dynamic program.   Let  agent $i \in \mathcal{N}^k$  of  sub-population $k \in \mathcal{K}$   at  time $t \in \mathbb{N}_T$ take the following action:
\begin{equation}\label{eq:mfs-optimal-finite-strategy}
u^i_t=g^{\ast,k}_t(x^i_t, \mathbf d_t):=\psi^{*,k}_t(\mathbf d_t)(x^i_t), \quad x^i_t \in \mathcal{X}^k, \mathbf d_t \in \MeanSpace.
\end{equation} 
A preliminary version of the next theorem was presented in~\cite{JalalCDC2014}.
\begin{Theorem}\label{thm:mfs-finite}
Let  Assumption~\ref{assumption:exchangeable} hold. Then,  strategy~\eqref{eq:mfs-optimal-finite-strategy} is  optimal for Problem~\ref{prob:MFS-finite}.
\end{Theorem}
\begin{proof}
The proof follows from the fact that the deep state $\mathbf d_t$ is an information state for Problem~\ref{prob:MFS-finite}. In particular,  according to Theorem~\ref{thm:finite-mfs-polynomial}, $\mathbf d_{1:T}$ is a controlled Markov process with control actions $\boldsymbol \gamma_{1:T-1}$. Furthermore,  according to equations~\eqref{eq:hat-l-def} and~\eqref{eq:joint mean-field-proof-1}, the per-step cost is a function of $\mathbf d_t$ and $\boldsymbol \gamma_t$.  Thus, the dynamic programming decomposition follows from standard results in Markov decision theory~\cite{kumar2015stochastic}. 
\end{proof}
The cardinality of   $\mathcal{E}_{|\mathcal{N}^k|}(\mathcal{X}^k)$  is upper-bounded polynomially in the number of the agents of sub-population~$k~\in~\mathcal{K}$, i.e.,
\begin{equation}\label{eq:mean-field-cardinality}
| \mathcal{E}_{|\mathcal{N}^k|}(\mathcal{X}^k) | \leq (|\mathcal{N}^k|+1)^{|\mathcal{X}^k|}.
\end{equation}
Consequently,  the space of  dynamic program~\eqref{eq:thm-dp-mfs-finie-t} increases at most polynomially with respect to  the number of agents in each sub-population $k \in \mathcal{K}$, and is independent of  time.
\begin{Remark}
\emph{The decomposition proposed in equation~\eqref{eq:thm-dp-mfs-finie-t} is a non-standard dynamic program because the minimization is over  local law $\boldsymbol \gamma_t \in \mathcal{G}$ rather than joint control action $\mathbf u_t \in \mathcal{U}$.} 
\end{Remark}
The  optimal strategy of Theorem~\ref{thm:mfs-finite} can be implemented in a distributed manner since every agent can independently compute the dynamic program \eqref{eq:thm-dp-mfs-finie-t}  and  observe the deep state.\footnote{In the case of multiple minimizers, agents agree upon a deterministic rule to choose one using $\argmin$.}  According to \eqref{eq:mfs-optimal-finite-strategy}, for any $k \in \mathcal{K}$, $i \in \mathcal{N}^k$ and $\NT$, the action (role) of agent $i$ of sub-population $k $ at time $t$ is determined by three factors:
\begin{enumerate}
 \item  global law $\boldsymbol \psi_t$ that depends on the agents'  dynamics, per-step cost, and underlying probability distributions;
 \item  deep state $\mathbf d_t$ that provides the statistical information on the states of agents,  and
 \item  local state $x^i_t$ that is private information for agent $i$ and unknown to others.
\end{enumerate}

In general, randomization can improve the performance of a fair strategy~\cite{Schoute1978}. On the other hand, any   parametrized randomized strategy  can be formulated as  a deterministic one wherein the randomization is embedded into the  transition probability and cost function, and the action space  is replaced  by the parameter space.  Therefore, finite parametrization does  not lead to an enhancement of the formulation.

\begin{Corollary}\label{remark:computational-space-time}
 The computational complexity of solving the dynamic program~\eqref{eq:thm-dp-mfs-finie-t} in both space and time is polynomial with respect to the number of agents in each sub-population $|\mathcal{N}^k|, k \in \mathcal{K},$ and is linear with respect to the control horizon~$T \in \mathbb{N}$.
\end{Corollary}
\begin{proof}
The proof is presented in Appendix~\ref{proof:remark:computational-space-time}.
\end{proof}
\begin{Remark}
\emph{It is to be noted that the fairness of  admissible strategies is essential  for establishing  Corollary~\ref{remark:computational-space-time}.  In addition, although the computational complexity of dynamic program~\eqref{eq:thm-dp-mfs-finie-t}  is  polynomial  with respect to  the number of agents, it is  exponential with  respect to the cardinality of the local state space,   which means  that the proposed strategy is only  tractable  for  small state spaces.   When the model  has a special structure, however,   it is possible to  relax the  restriction to fair strategies as well as  to  small  state spaces. Such special structures include, for example,   linear quadratic model  wherein the state space is an infinite set and   the optimal strategy is not necessarily exchangeable~\cite{Jalal2019risk}.}
\end{Remark}


For any realization $\mathbf d_t \in \MeanSpace$, local law $\boldsymbol \gamma_t \in \mathcal{G}$,   sub-population  $ k \in \mathcal{K}$,  and   pair of states  $x,y \in \mathcal{X}^k$  at time $\NT$, define $ B^k_t(y,x, \mathbf d_t, \boldsymbol \gamma_t) \in \mathcal{P}(\{0,1,\ldots,|\mathcal{N}^k|d^k_t(x)\})$ as:
\begin{multline}\label{eq:bionomial-iid}
B^k_t(y,x, \mathbf d_t, \boldsymbol \gamma_t):= \ID{d^k_t(x)=0} \delta_0 + \ID{d^k_t(x) > 0}   \\
\times  \Binopdf\left(\boldsymbol \cdot, |\mathcal{N}^k|d^k_{t}(x), \ProbK{y \mid x, \gamma^k_t(x),\phi(\mathbf d_t,\boldsymbol \gamma_t)}\right),
\end{multline}
where  $\mathbb{P}^k$ and $\phi$  are given by~\eqref{eq:relation-models} and~\eqref{eq:joint mean-field-proof-1}, respectively.  Furthermore, define $\bar B^k_t(y, \mathbf d_t, \boldsymbol \gamma_t) \in  \mathcal{P}(\{0,1,\ldots,|\mathcal{N}^k|\})$  as the convolution of the vector-valued functions $B^k_t(y,x, \mathbf d_t, \boldsymbol \gamma_t)$  over all states $x\in \mathcal{X}^k=\{ x_1,\ldots,x_{|\mathcal{X}^k|}\}$, i.e., 
 \begin{equation}\label{eq:iid-B-conv}
\bar B^k_t(y,\mathbf d_t, \boldsymbol \gamma_t):= B^k_t(y,x_1, \mathbf d_t, \boldsymbol \gamma_t) \ast \ldots \ast  B^k_t(y,x_{|\mathcal{X}^k|}, \mathbf d_t, \boldsymbol \gamma_t).
\end{equation}
\begin{Theorem}\label{thm:mfs-finite-independent}
Let  Assumption~\ref{assumption:iid} hold.  The transition probability matrix of  the deep state  can be computed  efficiently in time  such that for any $\mathbf d_t \in \MeanSpace$, $\boldsymbol \gamma_t \in \mathcal{G}$, $k \in \mathcal{K}$,   $y \in \mathcal{X}^k$  and  $j \in \mathbb{N}_{|\mathcal{N}^k|+1}$ at time $\NT$, we have:
\begin{equation}\label{eq:DCK}
 \Prob{d^k_{t+1}(y)= \frac{j-1}{|\mathcal{N}^k|}\mid \mathbf d_t, \boldsymbol \gamma_t}=\bar  B^k_t(y,\mathbf d_t, \boldsymbol \gamma_t)(j).
\end{equation}
\end{Theorem} 
\begin{proof}  
From  equations~\eqref{eq:joint mean-field-proof-1} and~\eqref{eq:mfs-proof-poly-1},  it results that for  any $\mathbf d_t \in \MeanSpace$, $\boldsymbol \gamma_t \in \mathcal{G}$, $k \in \mathcal{K}$ and $y \in \mathcal{X}^k$  at time $\NT$:
\begin{multline}\label{eq:iid-proof-1}
|\mathcal{N}^k| d^k_{t+1}(y)=  \sum_{x \in \mathcal{X}^k}\SUMI \ID{x^i_t=x}\\
\times \ID{f^k_t(x,\gamma^k_t(x),\phi(\mathbf d_t,\boldsymbol \gamma_t),w^i_t)=y}.
\end{multline}
For any sub-population  $k \in \mathcal{K}$ and any pair of states $x,y \in \mathcal{X}^k$  at time $t \in \mathbb{N}_T$, define $\tilde B^k_t(y,x, \mathbf d_t, \boldsymbol \gamma_t) \in  \mathcal{P}(\{0,1\})$ as the probability distribution function of the binary random variable $\ID{f^k_t(x,\gamma^k_t(x),\phi(\mathbf d_t,\boldsymbol \gamma_t),w^i_t)=y}$. From  \eqref{eq:dynamics-f-mean-field}, \eqref{eq:relation-models}, and~\eqref{eq:gamma-u}, one has
\begin{multline}\label{eq:proof-iid-1}
\Prob{\ID{f^k_t(x,\gamma^k_t(x),\phi(\mathbf d_t,\boldsymbol \gamma_t),w^i_t)=y}=1 \mid  \mathbf d_t, \boldsymbol \gamma_t} \\
=\ProbK{y \mid x, \gamma^k_t(x),\phi(\mathbf d_t,\boldsymbol \gamma_t)}.
\end{multline}
Thus,   it follows  from~\eqref{eq:proof-iid-1} that
\begin{multline}\label{eq:proof-iid-b}
\tilde B^k_t(y, x, \mathbf d_t, \boldsymbol \gamma_t)=  (1 - \ProbK{y \mid x, \gamma^k_t(x),\phi(\mathbf d_t,\boldsymbol \gamma_t)})\delta_0 \\
+ \ProbK{y \mid x, \gamma^k_t(x),\phi(\mathbf d_t,\boldsymbol \gamma_t)}\delta_1.
\end{multline}
 Let  $ B^k_t(y,x, \mathbf d_t, \boldsymbol \gamma_t) \in \mathcal{P}(\{0,1,\ldots,|\mathcal{N}^k|d^k_t(x)\})$ denote  the probability distribution function of the sum  of $|\mathcal{N}^k| d^k_t(x)$ binary random variables associated with state $x$ in~\eqref{eq:iid-proof-1}, i.e., 
$B^k_t(y,x, \mathbf d_t, \boldsymbol \gamma_t) =\mathbb{P}(\SUMI \ID{f^k_t(x,\gamma^k_t(x), \phi(\mathbf d_t,\boldsymbol \gamma_t),w^i_t)=y})$.
Therefore, if $d^k_t(x)=0,$ it means $B^k_t(y,x, \mathbf d_t, \boldsymbol \gamma_t)=\delta_0$; otherwise,  from equation~\eqref{eq:proof-iid-b}, Assumption~\ref{assumption:iid}, and the fact that the probability distribution function of the sum of  independent random variables is  equal to the convolution of their probability distribution functions, we have
\begin{equation}\label{eq:proof-iid-barb}
B^k_t(y,x, \mathbf d_t, \boldsymbol \gamma_t) =\underbrace{\tilde B^k_t(y, x, \mathbf d_t, \boldsymbol \gamma_t)\ast \ldots \ast \tilde B^k_t(y, x, \mathbf d_t, \boldsymbol \gamma_t)}_{|\mathcal{N}^k| d^k_t(x)}.
\end{equation}
Let  $\bar B^k_t(y,\mathbf d_t, \boldsymbol \gamma_t) \in  \mathcal{P}(\{0,1,\ldots,|\mathcal{N}^k|\})$   be the probability distribution  function of \eqref{eq:iid-proof-1}  (that is the sum of  $|\mathcal{N}^k|$ independent binary random variables); hence,  it is the convolution of  $B^k_t(y,x, \mathbf d_t, \boldsymbol \gamma_t) $ over all states $x\in \mathcal{X}^k$.  Notice that  $ B^k_t(y,x, \mathbf d_t, \boldsymbol \gamma_t)$ is an $|\mathcal{N}^k|d^k_t(x)-$fold convolution power of $\tilde B^k_t(y, x, \mathbf d_t, \boldsymbol \gamma_t)$, where $\tilde B^k_t(y, x, \mathbf d_t, \boldsymbol \gamma_t)$ is a binomial probability distribution function with success probability $\ProbK{y \mid x, \gamma^k_t(x), \phi(\mathbf d_t,\boldsymbol \gamma_t)}$.  Thus,  $ B^k_t(y,x, \mathbf d_t, \boldsymbol \gamma_t)$ is also a binomial probability distribution  with $|\mathcal{N}^k| d^k_{t}(x)$ trials and success probability $\ProbK{y \mid x, \gamma^k_t(x),\phi(\mathbf d_t,\boldsymbol \gamma_t)}$,  given by~\eqref{eq:bionomial-iid}. 
\end{proof}
\subsection{Deep Chapman-Kolmogorov (DCK) equation}\label{remark:chapman}
According to Theorem~\ref{thm:mfs-finite-independent}, the structure of  the transition probability matrix of  the deep state  involves multiple  convolution functions in~\eqref{eq:iid-B-conv}  and several Binomial probability distributions in~\eqref{eq:bionomial-iid},  triggered by some activation functions.
This structure resembles  a deep  neural network with convolutional layers and  Gaussian filters, which obeys a similar invariance feature called spatial invariance. Inspired by this resemblance, we refer to~\eqref{eq:DCK} as the \emph{Deep Chapman-Kolmogorov} (DCK) equation, which can be traced back to  the  seminal work  of  Bogolyubov and Krylov (see~\cite{bogolyubov1994nn}). In what follows,   we present some of  the special cases of the  DCK equation. 

 Consider a homogeneous population with no control action, where $\mathcal{K}$ is a singleton set. The following  holds:
\begin{itemize}
\item For the single-agent case,  the local state  and   deep state have equivalent information, meaning that the single-agent DCK can be represented equivalently in the form  of the (classical) Chapman-Kolmogorov equation.
\item  For the infinite-population case with coupled dynamics,   deep state  $d_t$ appears  in the transition probability matrix as a diffusion term; hence, the macroscopic DCK equation  may be viewed as the   discrete-time discrete-space  counterpart of  the Fokker-Planck-Kolmogorov  and master equations  in the continuous-time models. For the special case of decoupled dynamics,  the macroscopic DCK  equation simplifies to the classical Chapman-Kolmogorov equation, i.e.,  $d_{t+1}(y) =\sum_{x \in \mathcal{X}} d_t(x) \Prob{y \mid x}=\sum_{x \in \mathcal{X}} \Prob{y \mid x}  \Prob{x \mid z}$, 
  where $d_{t-1}:=\delta_z$, $y,z \in \mathcal{X}$. Consequently,  the  limit of  deep state $d_{t+1}(\boldsymbol \cdot)$,  as the population grows, which is in fact,  mean-field,   can be interpreted as the infinite-sample distribution of a generic  random variable $x_{t+1}$ with the  transition probability $\Prob{x_{t+1}=y|x_{t-1}=z}$.
\end{itemize}
For more details on the connection with deep neural networks, the interested  reader is referred to~\cite{Jalal2019risk,Jalal2019TSNE}.

\section{Main results for Problem~\ref{prob:PMFS-finite}}\label{sec:main_results_pmfs}
So far, we have assumed that the deep states of all sub-populations are shared among agents. However, for large sub-populations it might be difficult to collect and share their deep states among agents. In such cases, we are interested in  Problem~\ref{prob:PMFS-finite}, where the deep states of some sub-populations $\mathcal{S}^c \subseteq \mathcal{K}$ are not observed. We propose a sub-optimal   strategy,  where the optimality gap converges to zero as the size of sub-populations $\mathcal{S}^c$ goes to infinity.  We impose Assumptions \ref{assumption:iid} and \ref{assumption:iid-x} on the primitive random variables and the following two  assumptions on the model.   
\begin{Assumption}\label{assumption:decoupled}
The~deep states~of~sub-populations~$\mathcal{S}^c$ do not affect the dynamics of~the~agents~of~sub-populations~$\mathcal{S}$.
\end{Assumption}
\begin{Remark}
\emph{
Notice  that the deep states of  sub-populations~$\mathcal{S}$ can affect the dynamics of   the agents of sub-populations $\mathcal{S}^c$. In addition,
 Assumption~\ref{assumption:decoupled} automatically holds for $\mathcal{S}=\emptyset$ (i.e., when the information structure is completely decentralized). }
\end{Remark}
An immediate implication of Assumption~\ref{assumption:decoupled} is that  the following relationship  holds for  $\bar f^k_t$ in~\eqref{eq:def-bar-f-k} for any  $k \in \mathcal{S}$,  $\NT$, $\mathbf z \in \IX, \boldsymbol \gamma \in \mathcal{G}$ and  $\underline{\mathbf w}_t \in \underline{\mathcal{W}}$:
\begin{equation}\label{eq:pmfs-assumption-decoupled}
\bar f^k_t((z^k)_{k \in \mathcal{S}},\boldsymbol  \gamma,\MW)=\bar f^k_t(\mathbf z,\boldsymbol  \gamma,\MW).
\end{equation}
 \begin{Assumption}\label{assumption:Lipschitz-finite}
For every $x,y \in \mathcal{X}^k$, $u \in \mathcal{U}^k$, $k \in \mathcal{K}$  and $\Bz_1, \Bz_2 \in  \prod_{k \in \mathcal{K}} \mathcal{I}(\mathcal{X}^k \times \mathcal{U}^k)$, there exist positive real constants $\LipK{1}_t,\Lip{2}_t$ (independent of $|\mathcal{N}^k|$), such that
%
\begin{align}
\big|\ProbK{y\mid x,u,\Bz_1} - \ProbK{y\mid x,u,\Bz_2}  \big| &\leq \LipK{1}_t\Ninf{\Bz_1 -\Bz_2},\\
\big| \bar c_t(\Bz_1) -\bar c_t(\Bz_2)\big| &\leq \Lip{2}_t\Ninf{\Bz_1 -\Bz_2}.
\end{align}
 \end{Assumption}
 \begin{Remark}
\emph{Any  polynomial function of $\Bz$ is a Lipschitz  function  in $\Bz$ because $\Bz$ is confined to a bounded interval.   It is worth highlighting that  any continuous function can be  approximated as closely as desired by polynomial functions  according to the Weierstrass Approximation Theorem. }
 \end{Remark}
To distinguish from the optimal solution under DSS, we use superscript $p$ to denote  the parameters associated with the sub-optimal  solution under PDSS. Denote by $\mathbf g^p$ the  sub-optimal strategy  and by $\mathbf d_t^p \in \MeanSpace$  the deep state of the system under strategy $\mathbf g^p$ at $\NT$. Similar to Section~\ref{sec:main-result-problem-finite},  we split the strategy $\mathbf{g}^p$ into two parts as follows.  For every sub-population $k \in \mathcal{K}$ and any  $(d^{p,k}_{1:t})_{k \in \mathcal{S}}$, $t \in \mathbb{N}_T$, define:
\begin{equation} \label{eq:psi-gamma-pmfs}
\begin{cases}
\psi^{p,k}_t((d^{p,k}_{1:t})_{k \in \mathcal{S}}):= g^{p,k}_t(\boldsymbol \cdot,(d^{p,k}_{1:t})_{k \in \mathcal{S}}),\\
\gamma^{p,k}_t:=\psi^{p,k}_t(( d^{p,k}_{1:t})_{k \in \mathcal{S}}),
\end{cases}
\end{equation}
where   $\gamma^{p,k}_t:\mathcal{X}^k \rightarrow \mathcal{U}^k$, $k \in\mathcal{K}$.  Therefore, the control action of agent $i \in \mathcal{N}^k$ of sub-population~$k \in \mathcal{K}$ under PDSS  is given by:
\begin{equation}\label{eq:gamma-u-pmfs}
u^i_t=g^{p,k}_t(x^i_t,(d^{p,k}_{1:t})_{k \in \mathcal{S}})=  \psi^{p,k}_t(( d^{p,k}_{1:t})_{k \in \mathcal{S}})(x^i_t)=\gamma^{p,k}_t(x^i_t).
\end{equation}
Notice that the local laws  in~\eqref{eq:gamma-u-pmfs} are different from those   in~\eqref{eq:gamma-u} because the information structures are different. However, the finite space $\mathcal{G}$ is the same for both DSS and PDSS, i.e., $\boldsymbol{\gamma}^p_t:=\{\gamma^{p,1}_t,\ldots,\gamma^{p,K}_t\} \in \mathcal{G}$. Since Theorem~\ref{thm:finite-mfs-polynomial} holds regardless of the global laws $\boldsymbol \psi_{1:t}$ (Remark~\ref{remark:control_law_independence}), it  also holds for $(\psi^{p,1}_{1:t}, \ldots, \psi^{p,K}_{1:t})$  as a restriction function of $\boldsymbol \psi_{1:t}$.   Hence,  the dynamics of the deep state of sub-population $k \in \mathcal{K}$ at time $\NT$ can be written as follows, according to Theorem~\ref{thm:finite-mfs-polynomial}:
\begin{equation}\label{eq:definition-bar f-hat-m-K}
 {d}^{p,k}_{t+1} \hspace{.05cm}  \substack{{a.s.}\\=}  \hspace{.05cm} \bar f^k_t(\mathbf{d}^p_t, \boldsymbol{\gamma}^p_t,
\MW).
\end{equation}
In the augmented form, we have
\begin{equation}\label{eq:definition-bar f-pmfs}
 \mathbf {d}^p_{t+1} \substack{{a.s.}\\=}\bar f_t(\mathbf{ d}^p_t, \boldsymbol{\gamma}^p_t, 
\underline{\mathbf w}_t).
\end{equation}
The remainder of this section is organized as follows. We first introduce a stochastic process that is controlled by local laws $\boldsymbol{\gamma}^p_{1:T}$ and is implementable under PDSS. Then,   we  propose a dynamic program based on this process and  show  that  its solution  is sub-optimal, by comparing its performance  with  the optimal solution of Theorem~\ref{thm:mfs-finite}.
\subsection{A mixed state}\label{sec:aux}  
Let $m^k_t \in \mathcal{P}(\mathcal{X}^k)$ denote the mean-field of sub-population $k \in \mathcal{S}^c$ at time $\NT$, where $|\mathcal{N}^k|$ is regarded as infinity. Define a mixed state $\mathbf p_t \in \DeltaS$ consisting of the deep states of  sub-populations $\mathcal{S}$ and the mean-fields of sub-populations $\mathcal{S}^c$, i.e., 
\begin{equation}\label{eq:aux-dynamics-2}
p^k_1:=d^{p,k}_1=d^k_1, \quad k \in \mathcal{S}, \quad \text{and} \quad  p^k_1:=m^k_1=\mathbb{P}_{X^k_1}, \quad k \in \mathcal{S}^c.
\end{equation}
Moreover,  for $t >1 $,  $\mathbf p_t  $ evolves under  $
\boldsymbol \gamma^p_t \in \mathcal{G}$ as follows:
\begin{equation}\label{eq:aux-dynamics-1}
p^k_{t+1}:=\begin{cases}
\bar f^k_t(\mathbf p_t,\boldsymbol \gamma^p_t , \MW), & k \in \mathcal{S}, \\
\FS^k_t(\mathbf p_t, \boldsymbol \gamma^p_t ), & k \in \mathcal{S}^c,
\end{cases}
\end{equation}
where $\bar f^k_t$  and  $\FS^k_t$ are  given by \eqref{eq:def-bar-f-k} and  \eqref{eq:hat-f-def}, respectively.  Denote by $\bar f^{\mathcal{S}}_t, \NT,$ the augmented form of \eqref{eq:aux-dynamics-1} as follows:
\begin{equation}\label{eq:def-tilde-f-pmfs}
\mathbf p_{t+1}=\bar f^{\mathcal{S}}_t(\mathbf p_t,\boldsymbol \gamma^p_t ,\underline{\mathbf w}_t).
\end{equation}
\begin{Proposition}\label{prop:adaptable-mean-field-pmfs}
 Under  Assumptions~\ref{assumption:iid}, \ref{assumption:iid-x}, and~\ref{assumption:decoupled},  the  stochastic process $\mathbf{p}_{1:t}$  is adapted to  the filtration generated by $(d^{p,k}_{1:t})_{k \in \mathcal{S}}$. In particular,  $\mathbf p_1$ is specified  by \eqref{eq:aux-dynamics-2},  and for  any  $\NT$: $
p^k_{t+1} \substack{a.s. \\=} d^{p,k}_{t+1}, k \in \mathcal{S}$ and  $
p^k_{t+1} = \FS^k_t( \mathbf{p}_t, \boldsymbol \gamma^p_t ), k \in \mathcal{S}^c$.
\end{Proposition}
\begin{proof}
The proof is presented in Appendix~\ref{sec:proof_prop:adaptable-mean-field-pmfs}.
\end{proof}

\subsection{Dynamic program for Problem~\ref{prob:PMFS-finite}}
We now propose  a dynamic program for Problem~\ref{prob:PMFS-finite}.  Define  real-valued functions $\{\VP_1,\ldots,\VP_T, \VP_{T+1}\}$ backward in time such that
$\VP_{T+1}(\mathbf{p})=~0$ for any $ \mathbf p \in \DeltaS.$ Also,
for any $\NT$ and  $\mathbf p \in \DeltaS$, define
\begin{equation}\label{eq:thm-dp-pmfs-finie-t}
\VP_t(\mathbf{p})=\min_{\boldsymbol \gamma \in \mathcal{G}} \big( \LS_t\left(\mathbf p, \boldsymbol \gamma \right) + \mathbb{E}[\VP_{t+1}(\bar f^{\mathcal{S}}_t(\mathbf p,\boldsymbol \gamma,\underline{\mathbf w}_t))]\big).
\end{equation}
The right-hand side of \eqref{eq:thm-dp-pmfs-finie-t} admits at least one minimizer because $\mathbf p_{t} \in \DeltaS$  can be expressed by $(d^{p,k}_{1:t})$ and $\boldsymbol \gamma^p_{1:t-1}$  according to Proposition~\ref{prop:adaptable-mean-field-pmfs}, implying that its  feasible set is finite.  As a result, the  minimization in \eqref{eq:thm-dp-pmfs-finie-t} is a search  over finite alphabets,  meaning that it always has a minimizer (solution). With a slight  abuse of notation,  let  $\boldsymbol{\psi}^p_t(\mathbf{p}):=\{\psi^{p,1}_t(\mathbf{p}),\ldots,  \psi^{p,K}_t(\mathbf{p})\}$  be one of the minimizers of the right-hand side of \eqref{eq:thm-dp-pmfs-finie-t}.  The proposed control law for agent $i \in \mathcal{N}^k$ of sub-population $k \in \mathcal{K}$ at time $t \in \mathbb{N}_T$ is given by:
\begin{equation}\label{eq:strategy-pmfs-sub-optimal}
u^i_t=g^{p,k}_t(x^i_t, \mathbf p_t):= \psi^{p,k}_t(\mathbf p_t)(x^i_t), \quad x^i_t \in \mathcal{X}^k, \mathbf p_{t} \in \DeltaS.
\end{equation}  
\begin{Remark}
\emph{An important feature of  the dynamic program~\eqref{eq:thm-dp-pmfs-finie-t} is that it is independent of the size of sub-populations $\mathcal S^c$, and so is the  strategy~$\mathbf g^p$. This is due to the fact that  $\DeltaS$, $\LS_t$ and $\bar f^{\mathcal{S}}_t$ given, respectively,  by    \eqref{eq:def-spaces-short}, \eqref{eq:hat-l-def}, and \eqref{eq:aux-dynamics-2}  are all independent of  $|\mathcal{N}^k|, k \in \mathcal S^c$.}
\end{Remark}

Compared to the DSS value function~\eqref{eq:thm-dp-mfs-finie-t} and strategy~\eqref{eq:mfs-optimal-finite-strategy}, the source of error  in the proposed  PDSS value function~\eqref{eq:thm-dp-pmfs-finie-t} and strategy~\eqref{eq:strategy-pmfs-sub-optimal} is the substitution of   the  model with dynamics $\bar f_{1:T}$ and deep state process $\mathbf d_{1:T}$  by a model with dynamics $\bar f^{\mathcal{S}}_{1:T}$ and mixed-state process $\mathbf p_{1:T}$,  respectively. In what follows,  we  show that the propagation of   such error  in the dynamics and value functions are bounded by  Lipschitz gains related to those  introduced  in Assumption~\ref{assumption:Lipschitz-finite}.

\subsection{Upper bounds on the  dynamics and value functions}\label{sec:lemmas}
\begin{Lemma}\label{lemma:lipschitz-hat-f-hat-l}
Let  Assumption~\ref{assumption:Lipschitz-finite} hold.  There exist positive constants $\Lip{3}_t, \Lip{4}_t$, $k \in \mathcal{K}$, $t \in \mathbb{N}_T$ (independent of $|\mathcal{N}^k|$),  such that for any $\mathbf z_1, \mathbf z_2 \in \IX$ and  $\boldsymbol \gamma \in \mathcal{G}$, we have
\begin{align}
\| \FS^k_t(\mathbf z_1, \boldsymbol \gamma) - \FS^k_t(\mathbf z_2, \boldsymbol \gamma)\| &\leq \Lip{3}_t \Ninf{ \mathbf z_1 - \mathbf z_2},\\
\big| \LS_t(\mathbf z_1, \boldsymbol \gamma) - \LS_t(\mathbf z_2, \boldsymbol \gamma)  \big| &\leq \Lip{4}_t \Ninf{ \mathbf z_1 - \mathbf z_2}.
\end{align}
\end{Lemma}
\begin{proof}
The proof is presented in Appendix~\ref{sec:proof_lemma:lipschitz-hat-f-hat-l}.
\end{proof}

\begin{Lemma}\cite[Lemma 2]{JalalCDC2017}\label{lemma:square-root-iid}
Consider a random vector $w_{1:n}$ consisting of $n \in \mathbb{N}$ i.i.d. random variables with probability function $P_W$. Then, 
$\Exp{\Ninf{\Emp{w_{1:n}} - P_W}} \leq \frac{1}{\sqrt n}.$
\end{Lemma}

\begin{Lemma}\label{lemma:lipschitz-bar-f-hat-f}
Let  Assumption~\ref{assumption:Lipschitz-finite} hold. For  any $\mathbf z_1, \mathbf z_2 \in \IX$, $\boldsymbol \gamma \in \mathcal{G}$ and   $k \in \mathcal{K}$ at time  $t \in \mathbb{N}_T$, we have
\begin{equation}
\Compress
\Exp{\Ninf{\bar f^k_t(\mathbf z_1,\boldsymbol \gamma, \MW) - \FS^k_t(\mathbf z_2,\boldsymbol \gamma)}} 
\leq \Lip{3}_t \Ninf{ \mathbf z_1 - \mathbf z_2} + \BigON.
\end{equation}
\end{Lemma}
\begin{proof}
The proof is presented in Appendix~\ref{sec:proof_lemma:lipschitz-bar-f-hat-f}. 
\end{proof}

\begin{Lemma}\label{lemma:lipschitz-bar-f-bar-f}
Let  Assumption~\ref{assumption:Lipschitz-finite} hold.   For  any $\mathbf z_1, \mathbf z_2 \in \IX$,  $\boldsymbol \gamma \in \mathcal{G}$ and $k \in \mathcal{K}$ at time  $t \in \mathbb{N}_T$, we have
\begin{equation}
\Exp{\Ninf{\bar f^k_t(\mathbf z_1,\boldsymbol \gamma, \MW) - \bar f^k_t(\mathbf z_2,\boldsymbol \gamma, \MW)}}
 \leq \Lip{3}_t \Ninf{ \mathbf z_1 - \mathbf z_2}.
\end{equation}
\end{Lemma}
\begin{proof}
The proof is presented in Appendix~\ref{sec:proof_lemma:lipschitz-bar-f-bar-f}.
\end{proof}

\begin{Lemma}\label{proposition:quantization-basic}
Let Assumption~\ref{assumption:Lipschitz-finite} hold. Given any function $f^j_t: \IX \times \mathcal{G} \times \underline{\mathcal{W}} \rightarrow \IX$,  $j \in \{1,2\},$   $\NT$,  define real-valued functions $\{V^j_{1}, \ldots, V^j_{T+1}\}$  such that  for any $\mathbf z^j \in \IX$,
$
V^j_{T+1}(\mathbf{z}^j)=0,
$
and for any $\NT$ and $\mathbf z^j \in \IX$, one has
\begin{equation}\label{eq:prop-D-t}
V^j_t(\mathbf{z}^j)=\min_{\boldsymbol \gamma^j \in \mathcal{G}} \big(\LS_t\left(\mathbf z^j,\boldsymbol \gamma^j\right) +\Exp{ V^j_{t+1}(f^j_t(\mathbf z^j, \boldsymbol \gamma^j,\underline{\mathbf w}_t)  )}\big).
\end{equation}
Suppose there exist scalars $\Lip{z}_t \in \mathbb{R}_{>0}$ and $\Delta \in \mathbb{R}_{\geq 0}$ such that for every $\mathbf z^1, \mathbf z^2 \in \IX$ and $\boldsymbol \gamma \in \mathcal{G}$, the following inequality holds:
\begin{equation}\label{eq:prop-z1-z2}
\Exp{\Ninf{ f^1_t(\mathbf z^1, \boldsymbol \gamma,\underline{\mathbf w}_t) - f^2_t(\mathbf z^2, \boldsymbol \gamma,\underline{\mathbf w}_t)}} \leq \Lip{z}_t \Ninf{ \mathbf z^1 - \mathbf z^2} + \Delta.
\end{equation}
 Then, at every time $\NT$, we have
\begin{equation}\label{eq:prop-D-difference}
| V^1_t (\mathbf z^1) -V^2_t (\mathbf z^2)| \leq \Lip{5}_t\Ninf{ \mathbf z^1 - \mathbf z^2} +\Lip{6}_t \Delta,
\end{equation}
where $\Lip{5}_{T+1}:=\Lip{6}_{T+1}:=0$, and for any  $\NT$,
 \begin{equation}\label{eq:H5H6}
 \Lip{5}_t:=\Lip{4}_t + \Lip{5}_{t+1}\Lip{z}_t, \quad  \text{ and } \quad  \Lip{6}_t:= \Lip{5}_{t+1}+\Lip{6}_{t+1}.
 \end{equation}
\end{Lemma}
\begin{proof}
The proof is presented in Appendix~\ref{sec:proof_proposition:quantization-basic}.
\end{proof}

\begin{Lemma}\label{proposition:quantization-basic-equality}
Consider Lemma~\ref{proposition:quantization-basic} without  the $\min$ operator in equation~\eqref{eq:prop-D-t}.  Then, inequality~\eqref{eq:prop-D-difference}  holds for any $\boldsymbol \gamma:=\boldsymbol \gamma^1=\boldsymbol \gamma^2$.
\end{Lemma}
\begin{proof}
The proof follows along the same lines as that of Lemma~\ref{proposition:quantization-basic} with no minimization operator.
\end{proof}

\begin{Lemma}\label{lemma:lipschitz-MF-Z}
Under Assumption~\ref{assumption:iid},  for any  $\mathbf d \in \MeanSpace$, $\mathbf{p} \in \DeltaS$ and  $\boldsymbol \gamma \in \mathcal{G}$  at time $\NT$, we have
\begin{equation}
\Exp{\Ninf{ \bar f_t(\mathbf d,\boldsymbol \gamma,\underline{\mathbf w}_t) - \bar f^{\mathcal{S}}_t (\mathbf p,\boldsymbol \gamma,\underline{\mathbf w}_t)} }  \leq  \Lip{3}_t \Ninf{\mathbf d- \mathbf p}+ \BigO,
\end{equation}
where $n=\min_{ k \in \mathcal{S}^c} |\mathcal{N}^k|$ in Problem~\ref{prob:PMFS-finite}.
\end{Lemma}
\begin{proof}
The proof follows  from equations \eqref{eq:bar-f-augmented-def} and \eqref{eq:aux-dynamics-1},  and Lemmas \ref{lemma:lipschitz-bar-f-hat-f} and \ref{lemma:lipschitz-bar-f-bar-f}.
\end{proof}

\subsection{Convergence result for Problem~\ref{prob:PMFS-finite}} 
Let $\mathbf g^\ast=\{(g^{\ast,k}_t)_{k \in \mathcal{K}}\}_{t=1}^T$  and  $\mathbf g^p=\{ (g^{p,k}_t)_{k \in \mathcal{K}}\}_{t=1}^T$ denote  the  strategies  proposed  in~\eqref{eq:mfs-optimal-finite-strategy} and~\eqref{eq:strategy-pmfs-sub-optimal}, respectively. Inspired by the  notion of the  \emph{price of information} in~\cite{bacsar2011prices}, we  define the price of information in a slightly different manner as follows.
\begin{Definition}[\textbf{Price of information}]
 The price of information (PoI) is defined as the loss of performance due to  using  the PDSS strategy $\mathbf g^p$ rather than the  DSS strategy $\mathbf g^\ast$, i.e.,
 \begin{equation}\label{eq:price_info}
 PoI:=| J_N(\mathbf{g}^p) - J_N(\mathbf g^\ast) |.
 \end{equation}
\end{Definition}

\begin{Theorem}\label{thm:pmfs-finite}
Let Assumptions~\ref{assumption:iid}, \ref{assumption:iid-x},  \ref{assumption:decoupled}, and \ref{assumption:Lipschitz-finite} hold.   The price of information~\eqref{eq:price_info} converges to zero at the rate $1/\sqrt n$, i.e., 
\begin{equation}\label{eq:epsilon-rate-pmfs}
 PoI  \leq  (\Lip{5}_1+\Lip{6}_1)\BigO,
\end{equation}
where $\BigO$  is independent of the control horizon~$T$, and the positive constants $\Lip{5}_1$ and $\Lip{6}_1$ are computed recursively   for $\NT$ from the Lipschitz constants of Lemma~\ref{lemma:lipschitz-hat-f-hat-l} as follows:
 $\Lip{5}_t:=\Lip{4}_t + \Lip{5}_{t+1}  \Lip{3}_t$ and $\Lip{6}_t:= \Lip{5}_{t+1}+\Lip{6}_{t+1}$,
where  $\Lip{5}_{T+1}:=\Lip{6}_{T+1}:=0$.  Hence, strategy~\eqref{eq:strategy-pmfs-sub-optimal}  is an $\varepsilon(n)$-optimal strategy for Problem~\ref{prob:PMFS-finite}.
\end{Theorem}
\begin{proof}
From the triangle inequality, $|J_N(\mathbf{g}^p) - J_N(\mathbf g^\ast)|$ is upper-bounded by
\begin{equation}\label{eq:inequality-triangle-pmfs-basic}
|J_N(\mathbf g^\ast) -\mathbb{E}[\VP_1(\mathbf p_1)]|+| J_N(\mathbf{g}^p) -\mathbb{E}[\VP_1(\mathbf p_1)]|.
\end{equation}
The first term in \eqref{eq:inequality-triangle-pmfs-basic} is associated with the difference between the optimal performance under DSS, given by Theorem~\ref{thm:mfs-finite}, and an approximate  cost-to-go function proposed  in~\eqref{eq:thm-dp-pmfs-finie-t}.   The upper bounds developed in Subsection~\ref{sec:lemmas} yield:
\begin{align}\label{eq:g-star-hat v-pmfs}
& |J_N(\mathbf g^\ast)- \mathbb{E}[ \VP_1(\mathbf{p}_1)]| \substack{(a)\\=}| \mathbb{E}[\VM_1(\mathbf d_1)]- \mathbb{E}[ \VP_1(\mathbf{p}_1)]| \nonumber \\
&\substack{(b) \\ \leq} \Exp{\big| \VM_1(\mathbf d_1) -\VP_1(\mathbf p_1) \big| } 
 \substack{(c) \\ \leq} \Lip{5}_1 \mathbb{E}\Ninf{ \mathbf d_1 - \mathbf p_1}+\Lip{6}_1\BigO \nonumber \\
&\substack{(d) \\ \leq} (\Lip{5}_1+\Lip{6}_1)\BigO,
\end{align}
where $(a)$ follows from Theorem~\ref{thm:mfs-finite}, i.e., $J_N(\mathbf g^\ast)=\mathbb{E}[\VM_1(\mathbf d_1)]$; $(b)$ follows from the fact  that for every random variable $ y \in \mathbb{R}$, $|\Exp{y}| \leq \Exp{|y|}$; $(c)$ follows from equations \eqref{eq:thm-dp-mfs-finie-t}, \eqref{eq:thm-dp-pmfs-finie-t},  and Lemmas~\ref{proposition:quantization-basic}  and~\ref{lemma:lipschitz-MF-Z}, where $\Lip{z}_t:= \Lip{3}_t$  and $\Delta=\BigO$, and $(d)$ follows from Assumption~\ref{assumption:iid-x}, equation \eqref{eq:aux-dynamics-2}, and Lemma~\ref{lemma:square-root-iid}. 

The second term in \eqref{eq:inequality-triangle-pmfs-basic} corresponds to the difference between the performance of the system  under the proposed strategy~\eqref{eq:strategy-pmfs-sub-optimal} and the approximate cost-to-go function defined  in~\eqref{eq:thm-dp-pmfs-finie-t}.  Let  $\boldsymbol \gamma^p_t=\boldsymbol \psi^p_t(\mathbf p_t)$ be  a minimizer of equation~\eqref{eq:thm-dp-pmfs-finie-t}; then, equation~\eqref{eq:thm-dp-pmfs-finie-t}  simplifies to:
 \begin{equation}\label{eq:proof-q-2-pmfs}
 \VP_1(\mathbf p_1)= \Exp{\sum_{t=1}^T \LS_t(\mathbf p_t, \boldsymbol \gamma^p_t) \mid \mathbf p_1}.
 \end{equation}
According to  equations  \eqref{eq:J-general-def}, \eqref{eq:hat-l-def},  \eqref{eq:proof-joint-mean-field-1} and~\eqref{eq:definition-bar f-pmfs},  the performance of  the proposed strategy~\eqref{eq:strategy-pmfs-sub-optimal}  is given by:
\begin{equation}\label{eq:proof-q-1-pmfs}
J(\mathbf g^p)=\Exp{ \sum_{t=1}^T \LS_t(\mathbf d^p_t, \boldsymbol \gamma^p_t)}.
\end{equation}
From equations \eqref{eq:definition-bar f-pmfs}, \eqref{eq:def-tilde-f-pmfs}, and Lemma~\ref{lemma:lipschitz-MF-Z}, the approximation error between processes $\mathbf d^p_{1:T}$ and  $\mathbf p_{1:T}$ is bounded as follows:  for any $\mathbf d^p_t \in \MeanSpace$ and $ \mathbf p_t \in \DeltaS$,
\begin{equation}\label{eq:hat j- hat v- 2-pmfs}
\Exp{\Ninf{\mathbf{d}^p_{t+1}- \mathbf p_{t+1}}}  \leq  \Lip{3}_t \Ninf{\mathbf{d}^p_{t}- \mathbf p_{t}}+\BigO, \quad \NT.
\end{equation}
 Therefore,  it results from \eqref{eq:proof-q-2-pmfs} and \eqref{eq:proof-q-1-pmfs}  that:
\begin{align}\label{eq:hat j-hat v-1-pmfs}
&| J_N(\mathbf{g}^p) - \mathbb{E}[\VP_1(\mathbf{p}_1) ]| =  | \mathbb{E}\big[\sum_{t=1}^T \LS_t(\mathbf d^p_t, \boldsymbol \gamma^p_t) - \sum_{t=1}^T \LS_t(\mathbf p_t,\boldsymbol \gamma^p_t)]| \nonumber\\
&\substack{(e) \\ \leq}  \Lip{5}_1 \mathbb{E} \| \mathbf{d}^p_1 - \mathbf p_1\| + \Lip{6}_1 \BigO \substack{(f) \\ \leq}  (\Lip{5}_1 + \Lip{6}_1) \BigO,
\end{align}
where $(e)$ follows from  equation~\eqref{eq:hat j- hat v- 2-pmfs}, Lemma~\ref{proposition:quantization-basic-equality}  (where $\Lip{z}_t:= \Lip{3}_t$  and $\Delta=\BigO$), and  the fact  that for every random variable $ y \in \mathbb{R}$, $|\Exp{y}| \leq \Exp{|y|}$, and $(f)$ follows from Assumption~\ref{assumption:iid-x}, equation \eqref{eq:aux-dynamics-2} and Lemma~\ref{lemma:square-root-iid}.  The proof of Theorem~\ref{thm:pmfs-finite} results from  \eqref{eq:inequality-triangle-pmfs-basic}, \eqref{eq:g-star-hat v-pmfs}, and \eqref{eq:hat j-hat v-1-pmfs}.
\end{proof}

\section{Quantization Results}\label{sec:quantized}
In this section,  we present quantized solutions for Problems~\ref{prob:MFS-finite} and~\ref{prob:PMFS-finite} to reduce  the complexity in  space  for   dynamic programs~\eqref{eq:thm-dp-mfs-finie-t} and~\eqref{eq:thm-dp-pmfs-finie-t}  in Theorems~\ref{thm:mfs-finite} and \ref{thm:pmfs-finite}, respectively.

 Given a subset $\mathcal{R} \subseteq \mathcal{K}$ of sub-populations and the number of quantization levels  $\qss \in \mathbb{N}$, we are interested to find an $\varepsilon(\qss)$-optimal solution for Problem~\ref{prob:MFS-finite} such that the computational complexity of finding this solution in space is independent of the size of sub-populations $\mathcal{R}$.  To distinguish this solution from the optimal one under DSS, we use superscript $q$ for the parameters associated with the quantized solution. Denote by $\mathbf g^q$ the strategy of the quantized solution and by $\mathbf d_t^q \in \MeanSpace$  the deep state of the system under strategy $\mathbf g^q$ at $\NT$, where $\mathbf d^q_1=\mathbf d_1$. We impose an assumption similar  to Assumption~\ref{assumption:decoupled}  as follows.
 \begin{Assumption}\label{assumption:decoupled-R}
Assumption~\ref{assumption:decoupled}  holds for  sub-populations $\mathcal{R}$, where  $\mathcal{S}^c$   is substituted by $\mathcal{R}$.
\end{Assumption}


 Define real-valued functions $\{\VQ_1,\ldots,\VQ_T, \VQ_{T+1}\}$ backward in time such that
$\VQ_{T+1}(\mathbf{q})=0$ for any $ \mathbf q \in \QR.$ Also, for any
 $\NT$ and  $\mathbf{q}\in \QR$, define
\begin{equation}\label{eq:thm-dp-mfs-finie-Q-t}
\VQ_t(\mathbf{q})=\min_{\boldsymbol{\gamma} \in \mathcal{G}} \left(\LS_t\left(\mathbf q, \boldsymbol \gamma \right) + \Exp{\VQ_{t+1}\left(Q(\bar f_t( \mathbf q, \boldsymbol{ \gamma}, \underline{\mathbf w}_t))\right)}\right).
\end{equation}
Denote by $\boldsymbol \psi^q_t(\mathbf{q})=\{\psi^{q,1}_t(\mathbf{q}),\ldots,\psi^{q,K}_t(\mathbf{q})\}$ an argmin of the right-hand side of \eqref{eq:thm-dp-mfs-finie-Q-t}.  Let  agent $i \in \mathcal{N}^k$ of sub-population $k \in \mathcal{K}$ at time $\NT$ take the following action:
\begin{equation}\label{eq:thm:q-local-law}
u^i_t=g^{q,k}_t(x^i_t,Q(\mathbf d^q_t)):=\psi^{q,k}_t(Q(\mathbf d^q_t))(x^i_t), \quad  x^i_t \in \mathcal{X}^k, \mathbf d^q_t \in \MeanSpace.
\end{equation}
Let also $\mathbf g^q=\{ (g^{q,k}_t)_{k \in \mathcal{K}}\}_{t=1}^T$ denote  the  strategy proposed  in~\eqref{eq:thm:q-local-law}.  We define the price of computation as follows.
\begin{Definition}[\textbf{Price of computation}]
 The price of computation (PoC) is defined as the loss of performance  due to using  the  strategy $\mathbf g^q$  rather than  the optimal strategy $\mathbf g^\ast$, i.e.,
 \begin{equation}\label{eq:price_computation}
 PoC:=| J_N(\mathbf{g}^q) - J_N(\mathbf g^\ast)|.
 \end{equation}
\end{Definition}

\begin{Theorem}\label{thm:mfs-quantized}
Let Assumptions~\ref{assumption:iid}, \ref{assumption:Lipschitz-finite}, and \ref{assumption:decoupled-R} hold.  The price of computation~\eqref{eq:price_computation}  converges to zero at the rate $1/ \qss$, i.e.,
\begin{equation}\label{eq:epsilon-rate-Q-mfs}
PoC  \leq  (\Lip{5}_1+\Lip{6}_1) \big(\frac{1}{\qss}\big),
\end{equation}
where $\Lip{5}_1$ and $\Lip{6}_1$ are given  in Theorem~\ref{thm:pmfs-finite}.  This implies that strategy~\eqref{eq:thm:q-local-law}  is an $\varepsilon(r)$-optimal strategy for Problem~\ref{prob:MFS-finite}.
\end{Theorem}
\begin{proof}
The proof follows along the same lines as that of Theorem~\ref{thm:pmfs-finite} with the distinction that the source of error   is  the quantization error, that is bounded by $\frac{1}{2r}$ from the definition of the quantizer function $Q$. Using the same  upper bounds developed in Subsection~\eqref{sec:lemmas},  it can be shown that the propagation of  the quantization error  over  the dynamics and value functions are  bounded by the same Lipschitz  gains  introduced in Theorem~\ref{thm:pmfs-finite}.
\end{proof}

In general, finding an exact solution to the dynamic program~\eqref{eq:thm-dp-pmfs-finie-t}  is  challenging  because set $\DeltaS$ is uncountable for $\mathcal{S} \neq \mathcal{K}$.  To address this challenge, we propose a quantized dynamic program similar to that in Theorem~\ref{thm:mfs-quantized}.   Let $\mathcal{R}$ in Theorem~\ref{thm:mfs-quantized} be equal to $\mathcal{S}^c$ in Theorem~\ref{thm:pmfs-finite}. We are interested in finding an $\varepsilon(n,\qss)$-optimal solution for Problem~\ref{prob:PMFS-finite} such that  not only the deep states of sub-populations $\mathcal{S}^c$ are not  communicated among agents, but also the computational complexity of finding such a solution is independent of the size of sub-populations $\mathcal{S}^c$.  We use superscript $pq$  for the parameters associated to this solution. 


 Define real-valued functions $\{\VPQ_1,\ldots,\VPQ_T, \VPQ_{T+1}\}$ backward in time such that
$\VPQ_{T+1}(\mathbf{q})=0$ for any  $\mathbf q \in \QS$.
Also,  for any $\NT$ and $\mathbf{q}\in \QS$, define
\begin{equation}\label{eq:thm-dp-pmfs-finie-Q-t}
\VPQ_t(\mathbf{q})=\min_{\boldsymbol{\gamma} \in \mathcal{G}} \left(\LS_t\left(\mathbf q, \boldsymbol \gamma \right) + \Exp{\VPQ_{t+1}\left(Q(\bar f^{\mathcal{S}}_t( \mathbf q, \boldsymbol{ \gamma},\underline{ \mathbf w}_t))\right)}\right).
\end{equation}
Denote by $\boldsymbol \psi^{pq}_t(\mathbf{q})=\{\psi^{pq,1}_t(\mathbf{q}),\ldots,\psi^{pq,K}_t(\mathbf{q})\}$ an argmin of the right-hand side of \eqref{eq:thm-dp-pmfs-finie-Q-t}. Let  agent $i \in \mathcal{N}^k$ of sub-population $k \in \mathcal{K}$ at time $t \in \mathbb{N}_T$ take the following action:
\begin{equation}\label{eq:thm:p-local-law}
g^{pq,k}_t(x^i_t,Q(\mathbf p_t)):=\psi^{pq,k}_t(Q(\mathbf p_t))(x^i_t), \quad x^i_t \in \mathcal{X}^k.
\end{equation}
\begin{Theorem}\label{thm:pmfs-quantized}
Let Assumptions~\ref{assumption:iid}, \ref{assumption:iid-x}, \ref{assumption:decoupled} and  \ref{assumption:Lipschitz-finite} hold. 
Then, $\mathbf g^{pq}:=\{(g^{pq,k}_t)_{k \in \mathcal{S}}\}_{t=1}^T$  is an $\varepsilon(n,r)$-optimal solution for Problem~\ref{prob:PMFS-finite}, i.e.,
\begin{equation}\label{eq:epsilon-rate-Q-pmfs}
\big| J(\mathbf g^{pq}) - J(\mathbf g^\ast) \big|\  \leq \varepsilon(n,\qss) = (\Lip{5}_1+\Lip{6}_1) ( \BigO+\big(\frac{1}{\qss}\big)),
\end{equation}
where   $\Lip{5}_1$, $\Lip{6}_1$, and  $\BigO$ are given in Theorem~\ref{thm:pmfs-finite}.  
\end{Theorem}
\begin{proof}
The proof follows  from Theorems~\ref{thm:pmfs-finite} and~\ref{thm:mfs-quantized}. To avoid repetition, the detailed proof is omitted.
\end{proof}

\begin{Corollary}\label{cor:rate}
If $r  \geq  \sqrt{n}$, then the rate of convergence of the solution of Theorem~\ref{thm:pmfs-quantized} is the same as that of Theorem~\ref{thm:pmfs-finite}.
\end{Corollary}

\section{Infinite horizon discounted cost }\label{sec:infinite-horizon}
In this section, we generalize our main results to the  infinite-horizon discounted cost.   It is assumed that the dynamics and per-step cost in Problems~\ref{prob:MFS-finite} and~\ref{prob:PMFS-finite} are \textit{time-homogeneous}, and that the information structures are the  same as those in Subsection~\ref{sec:admissible}.   Given  a discount factor $\beta \in (0,1)$, the performance of any strategy $\mathbf{g}$  is described by:
\begin{equation}
J^{\beta}_N(\mathbf g)=\mathbb{E}^{\mathbf{g}}\big[\sum_{t=1}^\infty \beta^{t-1} c(\boldsymbol{\mathfrak{D}}_t)\big]. 
\end{equation}

\begin{Theorem}\label{thm:mfs-finite-discounted}
Let  Assumption~\ref{assumption:exchangeable} hold. The optimal solution  of Problem~\ref{prob:MFS-finite} with  the infinite-horizon discounted cost  is obtained from the following Bellman equation, i.e.,  for any $\mathbf d \in \MeanSpace$:
\begin{equation}\label{eq:thm-dp-mfs-finie-discounted}
\VM(\mathbf{d})=\min_{\boldsymbol \gamma \in \mathcal{G}} \big(\LS(\mathbf d, \boldsymbol \gamma) +\beta \Exp{ \VM(\bar f(\mathbf d, \boldsymbol \gamma, \underline{\mathbf w}))}\big),
\end{equation}
where the above expectation is taken with respect to $\underline{\mathbf{w}} \in \underline{\mathcal{W}}$. Let $\boldsymbol \psi^{*}(\mathbf{d})=\{\psi^{*,1}(\mathbf{d}),\ldots,\psi^{*,K}(\mathbf{d})\}$ be an argmin of the right-hand side of \eqref{eq:thm-dp-mfs-finie-discounted}. Then,  the optimal control  law of agent $i \in \mathcal{N}^k$ of sub-population $k \in \mathcal{K}$ at time $t \in \mathbb{N}$ is given by:
\begin{equation}
g^{\ast,k}(x^i_t, \mathbf d_t):=\psi^{*,k}(\mathbf d_t)(x^i_t), \quad x^i_t \in \mathcal{X}^k, \mathbf d_t \in \MeanSpace.
\end{equation} 
\end{Theorem}
 \emph{Proof:} Consider  the dynamic program of Theorem~\ref{thm:mfs-finite} for any finite horizon~$T$. From \cite{kumar2015stochastic}, make a change of variable  for  any $\mathbf d \in \MeanSpace$ and  $\NT$ such that 
\begin{align}\label{eq:def-W-infinite}
W^d_t(\mathbf d)&:=\beta^{-T+t-1} \VM_{T-t+2}(\mathbf d),
\end{align}
where
$W^d_1(\mathbf d):=\beta^{-T} \VM_{T+1}(\mathbf d)=0.
$
By simple algebraic manipulations  and setting $t=1$, we arrive at:
$W^d_{T+1}(\mathbf d)=\min_{\boldsymbol \gamma \in \mathcal{G}} ( \LS (\mathbf d,\boldsymbol \gamma) + \beta \Exp{W^d_{T}(\bar f (\mathbf d, \boldsymbol \gamma,\underline{\mathbf w})) } )$.
Since the above Bellman operator is contractive~\cite{kumar2015stochastic},   we have
\begin{equation}\label{eq:V-infinite-mfs}
\lim_{T \rightarrow \infty} W^d_{T} =  W^d_{\infty}=:\VM.
\end{equation}
\begin{Remark}
\emph{Although the computational complexity of finding the solution of  Bellman equation~\eqref{eq:thm-dp-mfs-finie-discounted} is polynomial  with respect to the number of agents in each sub-population, it  is  exponential in time, in general.  However,   one can find an $\varepsilon$-optimal solution  in polynomial time by using the value iteration method,  where $\varepsilon$ converges  to zero exponentially~\cite{kumar2015stochastic}.}
\end{Remark}
For Problem~\ref{prob:PMFS-finite} with  the infinite-horizon discounted cost, we impose the following assumption on the model.
\begin{Assumption}\label{assumption: finite-betak_infinite-horizon}
It is  assumed that $\beta \Lip{3}< 1$,   where $\Lip{3}$ is  the Lipschitz constant  in Lemma~\ref{lemma:lipschitz-hat-f-hat-l}.
\end{Assumption}
When the dynamics of agents are decoupled, Assumption~\ref{assumption: finite-betak_infinite-horizon}  holds because $ \Lip{3}=1$ satisfies Lemma~\ref{lemma:lipschitz-hat-f-hat-l}. This is an immediate consequence of equation~\eqref{eq:hat-f-def} and the fact that the probability of any event is upper-bounded by  $1$.
\begin{Theorem}\label{thm:pmfs-finite-discounted}
Let Assumptions~\ref{assumption:iid}, \ref{assumption:iid-x},~\ref{assumption:decoupled},~\ref{assumption:Lipschitz-finite} and \ref{assumption: finite-betak_infinite-horizon} hold. Then,  an $\varepsilon(n)$-optimal solution for  Problem~\ref{prob:PMFS-finite} with  the infinite-horizon discounted cost  function is  identified by the following Bellman equation  for any $\mathbf p \in \DeltaS$:
\begin{equation}\label{eq:thm-dp-pmfs-finie-t-discounted}
\VP(\mathbf{p})=\min_{\boldsymbol \gamma \in \mathcal{G}} \big( \LS(\mathbf p, \boldsymbol \gamma ) + \beta \Exp{\VP(\bar f^{\mathcal{S}} (\mathbf p,\boldsymbol \gamma,\underline{\mathbf w}))}\big),
\end{equation}
where the above expectation is taken with respect to  $\underline{\mathbf{w}} \in \underline{\mathcal{W}}$. Let $\boldsymbol \psi^p(\mathbf p)=\{\psi^{p,1}(\mathbf{p}),\ldots,  \psi^{p,K}(\mathbf{p})\}$ be an argmin  of the right-hand side of \eqref{eq:thm-dp-pmfs-finie-t-discounted}. The control law of agent $i \in \mathcal{N}^k$ of sub-population $k \in \mathcal{K}$ at time $t \in \mathbb{N}_T$   is given by:
\begin{equation}\label{eq:strategy-pmfs-sub-optimal-infinite}
u^i_t=g^{p,k}(x^i_t, \mathbf p_t)=\psi^{p,k}(\mathbf p_t)(x^i_t), \quad x^i_t \in \mathcal{X}^k, \mathbf p_t \in \DeltaS,
\end{equation}
and the optimality gap is bounded as follows:
\begin{equation}\label{eq:epsilon-rate-pmfs-discounted}
| J^\beta_N(\mathbf{g}^p) - J^\beta_N(\mathbf g^\ast)|\  \leq \varepsilon(n) = \frac{\Lip{4}}{(1-\beta)(1-\beta \Lip{3})}\BigO,
\end{equation}
where $\Lip{3}$ and $ \Lip{4}$ are the Lipschitz constants in Lemma~\ref{lemma:lipschitz-hat-f-hat-l}.
\end{Theorem}
\begin{proof}
Consider the dynamic program  of Theorem~\ref{thm:pmfs-finite} for any finite horizon $T$. Make a change of variables for any $\mathbf p \in \DeltaS$ and $\NT$  such that 
\begin{multline} \label{eq:def-hat W-infinite}
 W^p_t(\mathbf p):=\beta^{-T+t-1} \VP_{T-t+2}(\mathbf p), \\
\hat H^5_t :=\beta^{-T+t-1} \Lip{5}_{T-t+2},  \hat H^6_t :=\beta^{-T+t-1} \Lip{6}_{T-t+2},
\end{multline}
where 
$W^p_1(\mathbf p):=\beta^{-T} \VP_{T+1}(\mathbf p)=0,
\hat H^5_{1}:=\beta^{-T} \Lip{5}_{T+1}=0$ and $
\hat H^6_{1}:=\beta^{-T} \Lip{6}_{T+1}=0.
$
From Theorem~\ref{thm:pmfs-finite}, one arrives at the following relations by simple algebraic manipulations and setting $t=1$:
\begin{equation}\label{eq:infintie-1}
\begin{cases}
W^p_{T+1}(\mathbf p)=\min_{\boldsymbol \gamma \in \mathcal{G}} \big( \LS(\mathbf p, \boldsymbol \gamma) + \beta \Exp{W^p_{T}(\bar f^{\mathcal{S}} (\mathbf p, \boldsymbol \gamma,\underline{\mathbf w})) } \big),  \\
\hat H^5_{T+1}=\Lip{4}+\beta \hat H^5_T \Lip{3}=\Lip{4} \sum_{\tau=1}^{T} (\beta \Lip{3})^{\tau-1}, \\
\hat H^6_{T+1}=\beta(\hat H^5_T+\hat H^6_T) \leq  \hat H^5_T\sum_{\tau=1}^T \beta^{T-\tau+1}.
\end{cases}
\end{equation}
Since the above Bellman operator is contractive~\cite{kumar2015stochastic},  we have
\begin{equation}\label{eq:hat V-infinite}
 \lim_{T \rightarrow \infty} W^p_{T} = W^p_{\infty}=:\VP.
\end{equation}
In addition, from Lemmas~\ref{proposition:quantization-basic} and~\ref{lemma:lipschitz-MF-Z} as well as equations \eqref{eq:def-W-infinite} and \eqref{eq:def-hat W-infinite}, for any $\mathbf d \in \MeanSpace$ and  $ \mathbf p \in \DeltaS$, we obtain
\begin{multline}\label{eq:inf-pmfs-proof}
\Ninf{W^d_{T+1}(\mathbf d) -W^p_{T+1}(\mathbf p) } =  \Ninf{\VM_1(\mathbf d) - \VP_1(\mathbf p) } \leq   \\
  H^5_1 \Ninf{\mathbf d - \mathbf p} +  H^6_1 \BigO = \hat H^5_{T+1}\Ninf{\mathbf d - \mathbf p} + \hat H^6_{T+1} \BigO.
\end{multline}
Therefore, from equations \eqref{eq:V-infinite-mfs}, \eqref{eq:hat V-infinite} and \eqref{eq:inf-pmfs-proof}, it results that
\begin{multline}\label{eq:infinite-2}
\Ninf{ W^d_\infty(\mathbf m_1) -W^p_\infty(\mathbf p_1) }= \Ninf{\VM(\mathbf d_1) -\VP(\mathbf p_1)} \\ \leq \hat H^5_\infty\Ninf{\mathbf d_1- \mathbf p_1} + \hat H^6_\infty \BigO,
\end{multline}
where from  Assumption~\ref{assumption: finite-betak_infinite-horizon} and equation \eqref{eq:infintie-1}:
$\hat H^5_\infty=\frac{\Lip{4}}{1-\beta \Lip{3}}$ and $\hat H^6_\infty \leq \frac{\beta \Lip{4}}{(1- \beta)(1-\beta \Lip{3})}$.
The rest of the proof follows along the same  lines as the proof of Theorem~\ref{thm:pmfs-finite}, starting from the triangle inequality, where inequality $(c)$ in \eqref{eq:g-star-hat v-pmfs} and inequality $(e)$ in \eqref{eq:hat j-hat v-1-pmfs} are replaced by inequality \eqref{eq:infinite-2}. 
\end{proof}
\begin{Remark}\label{remark:infinite-quantization}
\emph{Since $ \DeltaS$  is a polish space, there always exists a minimizer $\boldsymbol \psi^p(\mathbf p)$  for \eqref{eq:thm-dp-pmfs-finie-t-discounted}.  Under Assumption~\ref{assumption: finite-betak_infinite-horizon}, the dynamic programs in Theorems~\ref{thm:mfs-quantized} and \ref{thm:pmfs-quantized}  can be extended  to  the infinite-horizon case  by replacing  $(\Lip{5}_1+ \Lip{6}_1)$ with $\frac{\Lip{4}}{(1-\beta)(1-\beta \Lip{3})}$.}
\end{Remark}


\begin{Remark}\label{remark:infinite}
\emph{Although strategy $\mathbf g^p$ in Theorem~\ref{thm:pmfs-finite-discounted}   is stationary with respect to $(x^i_t,\mathbf p_t)$, it is not stationary with respect to $x^i_t$, $\NT$. This implies  that the assumption of stationary strategy in~\cite{tembine2009mean} is rather restrictive. 
}
\end{Remark}

\section{Arbitrarily asymmetric cost function}\label{sec:arbitrary}
Let $c_t(\mathbf x_t,\mathbf u_t):\mathcal{X} \times \mathcal{U} \rightarrow \mathbb{R}_{\geq 0}$ be any arbitrarily-coupled (asymmetric) per-step cost function at time $\NT$. 
\begin{Assumption}\label{assumption:exchangeable-x}
For  any sub-population $k \in \mathcal{K}$,  initial states $(x^i_1)_{i \in \mathcal{N}^k}$   are exchangeable.
\end{Assumption}
\begin{Proposition}\label{propos:arbitrary coupled}
Let Assumptions~\ref{assumption:exchangeable}  and~\ref{assumption:exchangeable-x} hold.  
When attention is restricted to  fair  strategies, there exists a function $\hat \ell_t:\MeanSpace \times \mathcal{G} \rightarrow \mathbb{R}$, independent of the global laws $\boldsymbol {\psi}_{1:t}$,  such that: $ \hat \ell_t(\mathbf d_t,\boldsymbol \gamma_t):=\Exp{c_t(\mathbf x_t,\mathbf u_t)\mid \mathbf d_{1:t},\boldsymbol \gamma_{1:t}}=\sum_{\mathbf x \in \mathcal{X}} \sum_{\mathbf u \in \mathcal{U}}$ $ \prod_{k \in \mathcal{K} } \prod_{i \in \mathcal{N}^k} \ID{u^i=\gamma^k_t(x^i)} c_t(\mathbf x,\mathbf u)  \Prob{\mathbf x_t=\mathbf x\mid \mathbf d_{1:t},\boldsymbol \gamma_{1:t}}.
$ 
\end{Proposition}
\begin{proof}
The proof follows from  a forward  induction proposed in~\cite[Lemma 2]{JalalCDC2014} which shows that $\Prob{\mathbf x_t=\mathbf x\mid \mathbf d_{1:t},\boldsymbol \gamma_{1:t}}$   is partially  exchangeable, i.e.,  it is representable by~$\mathbf d_t$. 
\end{proof}
According to Proposition~\ref{propos:arbitrary coupled},  dynamic programs  proposed in Sections~\ref{sec:main-result-problem-finite}--\ref{sec:infinite-horizon} extend naturally to  any arbitrary asymmetric  cost function.  Note that the complexity of computing  $\hat \ell_t(\mathbf d_t,\boldsymbol \gamma_t)$ in time  is exponential with respect to the number of agents. However, this computation can be carried out off-line  by machine learning methods or circumvented by reinforcement learning techniques~\cite{JalalACC2015,Jalal2019TSNE}. In general, the exploration space of  an arbitrary asymmetric cost function  grows exponentially with  the number of agents while that  of its  deep state projection grows polynomially, according to Proposition~\ref{propos:arbitrary coupled}, which is a considerable reduction in complexity.
 
\section{Major-minor setup: a special case}\label{sec:MM}
Consider a special case  where a sub-population has only  one agent. In mean-field game theory,  this  case is known as  \emph{major-minor mean-field game} which  was first  introduced  in~\cite{Huang2010large}. The sub-population with one agent is referred to as the \emph{major}  player because it  can directly influence  other  players, called \emph{minor} players,  through its local state whereas  other players can only influence the major player through their collective behaviour (mean-field).  In this type of scenario, the classical mean-field game approach  is not directly applicable   as the mean-field of minor players  is no longer deterministic (hence, unpredictable)  due to  the randomness of the major player's state~\cite{Nourian2013MM}. A similar setup  may  be considered in deep   teams with the  distinction that  no  additional complication is introduced as the deep team solution is not in the form of coupled forward-backward equations, i.e., it is independent of the future trajectory of the deep state. 
To connect our results to  major-minor setup,  consider a  sub-population $k \in \mathcal{K}$  consisting of  one agent, i.e., $|\mathcal{N}^k|=1$.  Since  the  local state  and  the deep state of this agent are  identical, the split of strategy  $\mathbf g^p$ in~\eqref{eq:psi-gamma-pmfs} can be done  slightly different as follows:
 $
u^i_t=\psi^{p,k}_t((d^{p,k}_{1:t})_{k \in \mathcal{S}})=g^{p,k}_t((d^{p,k}_{1:t})_{k \in \mathcal{S}})
$. 
In  this case,  the local  law of major agent  $\gamma^k_t: \mathcal{X}^k \rightarrow \mathcal{U}^k$ (that takes $|\mathcal{U}^k|^{|\mathcal{X}^k|}$ values)  simplifies to  its local control action  $u^i_t \in \mathcal{U}^k, i \in \mathcal{N}^k,$ (that takes $|\mathcal{U}^k|$ values).

\section{Numerical Example}\label{sec:numerical}
\textbf{Example 1.} Consider a company that provides a particular service  (e.g., internet, electricity or cellular phone) for $n \in \mathbb{N}$ users.  Assume that each  user $i \in \mathbb{N}_n$ at time $t \in \mathbb{N}$ makes an independent  request with probability $\mu \in (0,1)$ to receive  service.  For simplicity, it is also assumed that   each user is not allowed  to make a new request until its current request is  served.  Therefore,  the state of user $i$ at time~$t$ is binary, i.e., $x^i_t \in \mathcal{X}:=\{0,1\}$, where $x^i_t=1$  means that user $i$ has a request at time $t$, and $x^i_t=0$ means that it does not have a request at that time. The initial states of users are distributed identically and independently with respect to  the probability mass function $\mathbb{P}_X$. Let $d_t \in \mathcal{E}:=\{0, \frac{1}{n},\frac{2}{n},\ldots,1\}$ denote the empirical distribution of the requests at time $t$, i.e.,
$d_t=\frac{1}{n}\sum_{i=1}^n \ID{x^i_t=1}.$
 To serve its users, the company  has  $h \in \mathbb{N}$  options which depend on different  contracts and resources.
Define the following  terms for every $u \in \mathcal{U}
=\{1,\ldots,h\}$:
\begin{itemize}
\item  \textbf{Participation  rate:}  This is the probability according to which  a user is incentivized to delay its request.   The company may use different contracts and price profiles to incentivize users to postpone their requests.
For example in smart grids,  the independent service operator may offer discounts to motivate users to delay their demands during peak-load time.
 Denote this probability by $\alpha(u) \in [0,1]$ and assume that it  is independent of the request probability; hence, the probability of receiving a request from a participant user is $(1- \alpha(u)) \mu \leq \mu.$ 

\item  \textbf{Service rate:} This is the probability according to which a request  is  served at each time instant.  The company may use various suppliers (resources) to provide service; hence, the service rate  may not be the same for different suppliers.  Denote  this probability  by  $q(u)  \in (0,1]$. 

\item  \textbf{Base price:} This is the price of serving a user  when  it has  no request (i.e., $x=0$). This price is  used  for maintaining resources such as data storage and communications. Denote this price  by $c_B(u,1-d_t)\in \mathbb{R}_{\geq 0}$.

\item  \textbf{Service  price:}  This is the price of serving  a user when it has a request (i.e., $x=1$).  This  includes the price of service as well as the price of  maintenance (but may not be a simple sum of the two). Denote this price  by $c_S(u,d_t)\in \mathbb{R}_{\geq 0}$.
%
\end{itemize}
Let  $u^i_t \in \mathcal{U}$ denote the option assigned to user $i \in \mathbb{N}_n$  at time~$t$.  Then, the transition probability matrix of user $i$ under option $u^i_t$ at time $t$ can be written as follows:
\begin{equation}
\Prob{x^i_{t+1}\mid x^i_t, u^i_t}=\left[
\begin{array}{cc}
1- (1-\alpha(u^i_t)) \mu & (1-\alpha(u^i_t))\mu\\
q(u^i_t) & 1-q(u^i_t)
\end{array}
\right].
\end{equation}

Alternatively,  one can express the above-mentioned  dynamics  in the form of~\eqref{eq:dynamics-f-mean-field}, i.e., 
$x^i_{t+1}=(1-x^i_t) w^i(u^i_t) + x^i_t  \hat w^i_t(u^i_t)$,
where $w^i(u^i_t) , \hat w^i(u^i_t)  \in \{0,1\}$ are the underlying Bernoulli random variables  associated with  the rates of request, participation, and service such that:
$\Prob{w^i(u^i_t)=1}= (1-\alpha(u^i_t)) \mu$ and $ \Prob{\hat w^i(u^i_t)=1}=1-q(u^i_t).$
In addition,  the cost of serving user~$i$   at time~$t$ is given by:
$c(x^i_t,u^i_t,d_t)= (1-x^i_t)c_{B}(u^i_t,1-d_t)+x^i_t c_S(u^i_t, d_t)$.
\begin{figure}[b!]
\centering
\vspace{-0.6cm}
\scalebox{1}{
\includegraphics[trim={0cm 3cm 0cm 3cm},clip, width=\linewidth]{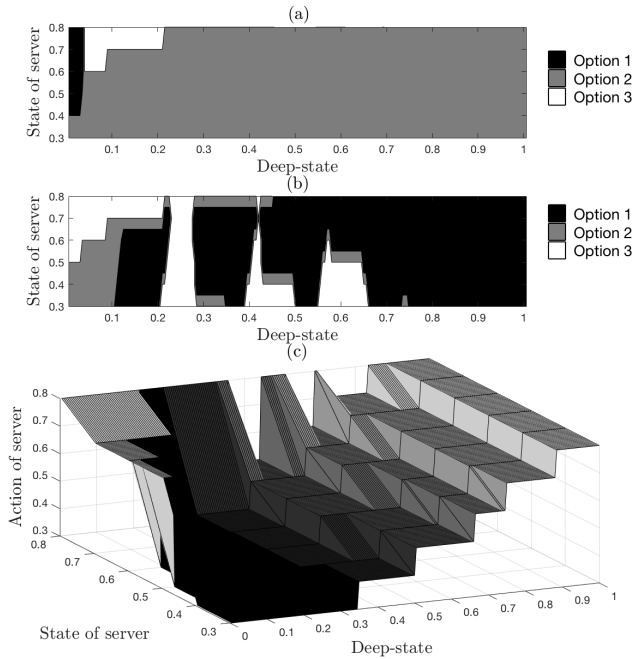}}
\caption{The optimal strategy  for  serving the users as well as  adjusting the nominal capacity in Example 1. (a) The optimal option  for  a user with no request  ($x=0$); (b) the optimal option for a user with request  ($x=1$), and $(c)$  The optimal nominal capacity for the server.}\label{fig1}
\end{figure}
\begin{figure}[t!]
\centering
\scalebox{.8}{
\includegraphics[trim={0cm 9.5cm 0 9cm},clip, width=\linewidth]{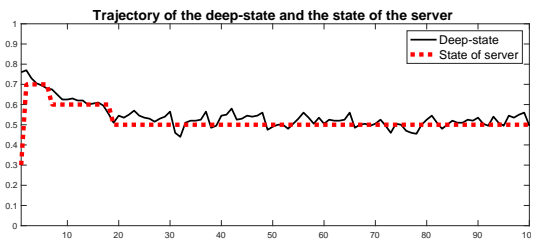}}
\caption{A trajectory  of the deep state (empirical distribution of requests) and  the state of the server (nominal capacity).}\label{fig:trajectory}
\end{figure}
 It is desirable  for the  company  to adapt its   nominal capacity with the empirical distribution of requests.  Let $x^0_t \in \mathcal{X}^0 $ denote the nominal capacity of    a server located in the  company at time $t \in \mathbb{N}$, where $\mathcal{X}^0$ is a finite set  with  values in $[0,1]$. At each time instant,  the  server  must decide    the next nominal capacity. Let   $u^0_t \in \mathcal{U}^0=:\mathcal{X}^0$ denote the action of  the  server at time $t \in \mathbb{N}$. Subsequently, the state  dynamics of the   server  evolves as
$x^0_{t+1}= x^0_tw^0_t+u^0_t (1-w^0_t)$,
where  $\{w^0_t \in \{0,1\}, t\in \mathbb{N}\}$ is a Bernoulli process that  captures  the  probability of failure, i.e., $w^0_t=1$ implies that there is a fault at time $t$,  and  $w^0_t=0$ implies that there is no fault.   Denote by  $p_{w^0} \in [0,1]$  the  probability according to which a  fault may occur.  When a fault happens,  no capacity is patched or dispatched.   Define $\ell: \mathcal{X}^0 \times \mathcal{U}^0 \rightarrow \mathbb{R}_{\geq 0}$ as  the cost function of  the server such that
$\ell(x^0_t,u^0_t):=\ell_C(x^0_t)+  \ell_P |u^0_t -x^0_t|,t \in \mathbb{N}$,
where $\ell_C(x^0_t) \in \mathbb{R}_{\geq 0 }$ is the price of capacity $x^0_t \in \mathcal{X}^0$  and $\ell_P |u^0_t -x^0_t|$, $\ell_P \in \mathbb{R}_{\geq 0}$,  is the cost  of patching and dispatching capacities $x^0_t$ and $u^0_t$. 

 The  objective of the company  is to manage  the users (consumption side)  as well  as the server (generation side) in such a way  that  the empirical distribution of requests is   close to the  nominal capacity  while incurring  the lowest possible  price.  More precisely, given $\beta \in (0,1)$ and $\lambda \in \mathbb{R}_{\geq 0}$,   it is desired to minimize the  following cost
\begin{equation}\label{eq:example-cost}
\mathbb{E}\big[\sum_{t=1}^\infty \beta^{t-1} \big(\frac{1}{n} \sum_{i=1}^nc(x^i_t,u^i_t,d_t) +\ell(x^0_t,u^0_t) + \lambda ( d_t - x^0_t)^2  \big)\big],
\end{equation}
where the first term  is the average price of users,  the second term is   the cost of the server, and the third  term is the penalty for  the empirical distribution of  requests   deviating from the  nominal capacity. Since  the state space $\mathcal{X} $ is binary,   the empirical distribution of requests $d_t$   is a  sufficient statistic to identify the deep state, i.e.,  $(1-d_t,d_t)$,  $t \in \mathbb{N}$.  Hence, to ease the presentation,  $d_t$ is  referred to as the deep state hereafter.

\begin{figure}[t!]
\centering
\scalebox{.8}{
\includegraphics[trim={0cm 9.5cm 0 9cm},clip, width=\linewidth]{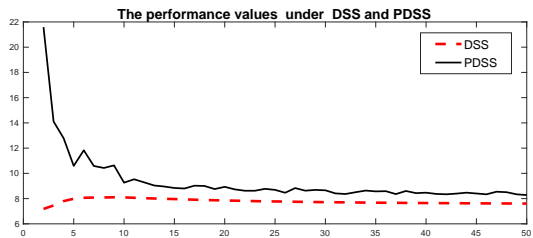}}
\caption{The convergence of the PDSS solution $J(g^{pq})$ with $n$ equidistant  quantization levels to the DSS solution $J(g^\ast)=\Exp{\VM(d_1,x^0_1)}$ as the number of users $n$ increases, according to  Theorem~\ref{thm:pmfs-quantized} and Remark~\ref{remark:infinite-quantization}. }\label{fig:MFS_PMFS}
\end{figure}
To illustrate the results, suppose the company has $3$ options ($h=3$). The first option, denoted by $u=1$,   is to provide service at a flat rate.  The second option, denoted by $u=2$, is to give $20\%$ discount  to users in order to incentivize them to  cooperate with the company.\footnote{The monetary discount is considered to  be an extra cost  for the company.}  In such  a case,  users are motivated to delay their requests (with a probability directly dependent on the participation rate)  and  let  the company  serve  their existing  requests  at a lower rate. The third option, denoted by $u=3$, is that the company purchases the desired service from another service company   at a variable rate.  In addition,  the  following  parameters are chosen for  the simulations:
\begin{align*}
&n=200, \quad \beta=0.8,\quad \mu=0.8, \quad q=[0.1, 0.05, 0.2],\\
 &\alpha=[0,0.85, 0],   \quad c_B(1,\boldsymbol \cdot)=0.59, \quad c_S(1,\boldsymbol \cdot)=0.65,\\
&c_B(2,\boldsymbol \cdot)=(1+0.2) c_B(1,\boldsymbol \cdot), \quad c_S(2,\boldsymbol \cdot)=(1+0.2)c_S(1,\boldsymbol \cdot), \\
 &c_B(3,1-d)=0.3(1+1-d),\quad  c_S(3,d)=0.5(1+d), \\
&\lambda=15, \quad \ell_P=0.5,\quad  \mathcal{X}^0=\mathcal{U}^0=\{0.3, 0.4, 0.5, 0.6, 0.7, 0.8\}\\
&\ell_C=[0.02,0.04,0.07,0.12,0.15,0.2],\quad p_{w^0}=0.05.
\end{align*}

  When there is no request  from user $i \in \mathbb{N}_n$ at time $t \in \mathbb{N}$,  the cheapest option for  the user  is chosen from   Figure~\ref{fig1}.a at that time, and  when there is a request  from that user,   its  best affordable  option   is selected from  Figure~\ref{fig1}.b.     In addition, the  nominal capacity of the server   is given by Figure~\ref{fig1}.c based on the current capacity and the deep state.  Figure~\ref{fig:trajectory} demonstrates  the  trajectory of  the deep state and  nominal capacity.
 In Figure~\ref{fig:MFS_PMFS},  the optimality gap between DSS and PDSS solutions with respect to the number of users can be observed, where the initial state of the server is $0.3$ and the initial states of users  are identically and independently distributed  according to the  probability distribution $[0.2, 0.8]$.


\section{Conclusions}\label{sec:conclusion}
In this paper,  we  introduced deep structured  teams  and proposed novel dynamic programs  to  find  optimal and sub-optimal strategies for agents in a team,  under  two non-classical information structures, namely, deep-state sharing and partial deep-state sharing. In addition, we  developed some  results to alleviate the computational complexity of the   proposed solutions when the population is medium or large.  We then extended our results to the  infinite-horizon discounted cost as well as arbitrarily coupled cost, and demonstrated  their effectiveness by a numerical example in service management.

In practice, agents often have limited  computation and communication  resources. These practical  limitations are the root causes of many challenges in team theory. We showed that in deep  teams,  the number of computational resources increases polynomially (rather than exponentially) with the number of agents in each sub-population. The size (length) of the shared information (i.e., deep state) among agents, on the other hand,  is independent of the number of agents in each sub-population.
Moreover,  agents may decide not to communicate their states  at all, in which case  the error of such compromise  was shown to be bounded by $\mathcal{O}(1/\sqrt{n})$ and hence goes to zero as the number of agents $n$ increases. Furthermore,  agents may  use a quantized solution whose computational complexity in space is independent of the  number of  agents. We showed that the error of such compromise is bounded by $\mathcal{O}(1/r)$, where  the error tends to zero as the number of quantization levels~$r$ increases. 
 
The main results of this paper can  naturally be generalized to randomized strategies and partially observable deep states by simply replacing the action space  and deep state with the space of  probability measures and  belief deep state, respectively. 
 To further enhance the computational complexity
of a deep team, one can use various approximation methods such as deep reinforcement learning. In such a case, the resultant design may be viewed as \emph{doubly deep} in the number of decision makers (that represent parallel computation) as well as the number of hidden layers (that represent sequential computation).

\bibliographystyle{IEEEtran}
\bibliography{Jalal_Ref}


\appendices

\section{Proof of Corollary~\ref{remark:computational-space-time}}\label{proof:remark:computational-space-time}
To numerically solve the dynamic program~\eqref{eq:thm-dp-mfs-finie-t}, at each step $\NT$,  it is required to store $|\MeanSpace|^2 |\mathcal{G}|$ entries corresponding to  the transition probability matrix of deep state \eqref{eq:mftpm}, and to run  $|\underline{\mathcal{W}}|$ iterations to compute each entry.  Each entry occupies  a certain amount of space and each iteration takes a certain amount of  time. In addition,   to solve the minimization problem  in \eqref{eq:thm-dp-mfs-finie-t}, $|\MeanSpace| |\mathcal{G}|$ iterations are  needed for searching over all possible cases. On the other hand, from~\eqref{eq:mean-field-cardinality}, $|\MeanSpace|$ and $|\underline{\mathcal{W}}|$ are bounded by 
$\prod_{k \in \mathcal{K}}(|\mathcal{N}^k|+1)^{|\mathcal{X}^k|}$ and  $\prod_{k \in \mathcal{K}}(|\mathcal{N}^k|+1)^{|\mathcal{W}^k|}$, respectively. Furthermore,  $|\mathcal{G}|=\prod_{k \in \mathcal{K}} |\mathcal{U}^k|^{|\mathcal{X}^k|}$ is finite and independent of the number of agents in each sub-population as well as  control horizon~$T$,  which completes to  the result.

\section{Proof or Proposition~\ref{prop:adaptable-mean-field-pmfs}}\label{sec:proof_prop:adaptable-mean-field-pmfs}
The proof proceeds by induction. According to equation \eqref{eq:aux-dynamics-2}, Proposition~\ref{prop:adaptable-mean-field-pmfs} holds at initial time $t=1$.
Suppose $\mathbf{p}_{1:t}$  is adapted to the filtration $(d^{p,k}_{1:t})_{k \in \mathcal{S}}$ at time $t$.  From Assumption~\ref{assumption:decoupled} and  equation  \eqref{eq:pmfs-assumption-decoupled}, the following equality holds for any $\mathbf d \in \MeanSpace$, $\mathbf p \in \DeltaS$ and $\boldsymbol \gamma \in \mathcal{G}$ at any time $\NT$, 
\begin{align}\label{eq:prop:proof_1}
\bar f^k_t((d^k)_{k \in \mathcal{S}},\boldsymbol  \gamma,\MW)&=\bar f^k_t(\mathbf d,\boldsymbol  \gamma,\MW), \quad k \in \mathcal{S}, \nonumber \\
\bar f^k_t((p^k)_{k \in \mathcal{S}},\boldsymbol  \gamma,\MW)&=\bar f^k_t(\mathbf p,\boldsymbol  \gamma,\MW), \quad k \in \mathcal{S}.
\end{align}
Consequently, processes $d^{p,k}_{1:t+1}$ and  $p^k_{1:t+1}$,  $k \in \mathcal{S}$,   start from identical initial values  and evolve  identically  in time  almost surely under identical local  laws and noise processes, according to  equations \eqref{eq:definition-bar f-hat-m-K},  \eqref{eq:aux-dynamics-2}, \eqref{eq:aux-dynamics-1}, \eqref{eq:prop:proof_1}, and  Assumption~\ref{assumption:decoupled}.  Thus, $p^k_{t+1}  \substack{a.s. \\=}  d^{p,k}_{t+1}$.
 For  any  $k \in \mathcal{S}^c$,  $p^k_{t+1}$ evolves deterministically   according to \eqref{eq:aux-dynamics-1},   where  $\mathbf p_t$ and  $\boldsymbol \gamma^p_t$ are adapted to $(d^{p,k}_{1:t})_{k \in \mathcal{S}}$ from the induction assumption and  \eqref{eq:psi-gamma-pmfs}, respectively.
  Therefore, $\mathbf{p}_{1:t+1}$  is adapted to  $(d^{p,k}_{1:t+1})_{k \in \mathcal{S}}$.
\section{Proof of Lemma~\ref{lemma:lipschitz-hat-f-hat-l}}\label{sec:proof_lemma:lipschitz-hat-f-hat-l}
It is well known that any linear combination and  product of Lipschitz functions is  also a Lipschitz function. According to~\eqref{eq:phi-def-function}, for any $ \boldsymbol \gamma  \in \mathcal{G}$, function $\phi(\mathbf z, \boldsymbol \gamma)$ is Lipschitz in~$\mathbf z$. Consequently, for any $y,x \in \mathcal{X}^k$ and any  $\boldsymbol \gamma \in \mathcal{G}$, the transition probability $\ProbK{y |x,\gamma^k(x),\phi(\mathbf z, \boldsymbol \gamma)}$ is Lipschitz  in~$\mathbf z$ due to Assumption~\ref{assumption:Lipschitz-finite}. Hence, function $\FS^k_t(\mathbf z, \boldsymbol \gamma)$, given by  \eqref{eq:hat-f-def}, is Lipschitz  in~$\mathbf z$ because  it is  a linear combination  of  Lipschitz functions (and their products). Let $\LipK{3}_t$ denote the  corresponding Lipschitz constant, and  define $\Lip{3}_t$ as the maximum  Lipschitz constant, i.e., $\Lip{3}_t:=\max_{k \in \mathcal{K}} \LipK{3}_t$.  Similarly, function $\LS_t(\mathbf z, \boldsymbol \gamma)$, given by  \eqref{eq:hat-l-def}, is Lipschitz  in $\mathbf z$ due to Assumption~\ref{assumption:Lipschitz-finite}, on noting that  the function $\phi$ is Lipschitz.


\section{Proof of Lemma~\ref{lemma:lipschitz-bar-f-hat-f}}\label{sec:proof_lemma:lipschitz-bar-f-hat-f}
For every $y \in \mathcal{X}^k$, $k \in \mathcal{K},$ it follows that
\begin{align}
 &\mathbb{E}\big[\big| \bar f^k_t(\mathbf z_1,\boldsymbol \gamma,\mathbf d_t)(y) -\FS^k_t(\mathbf z_2,\boldsymbol \gamma)(y)\big| \big]\\
&=\mathbb{E}\big[\big| \bar f^k_t(\mathbf z_1,\boldsymbol \gamma,\mathbf d_t)(y)  \pm \FS^k_t(\mathbf z_1,\boldsymbol \gamma)(y) -\FS^k_t(\mathbf z_2,\boldsymbol \gamma)(y)\big| \big]\\
& \substack{{(a)} \\ \leq }\Exp{\big| \bar f^k_t(\mathbf z_1,\boldsymbol \gamma,\mathbf d_t)(y)  - \FS^k_t(\mathbf z_1,\boldsymbol \gamma)(y)\big|}\\
& \qquad + \big| \FS^k_t(\mathbf z_1,\boldsymbol \gamma)(y)  -\FS^k_t(\mathbf z_2,\boldsymbol \gamma)(y)\big|\\
&\substack{{(b)} \\ \leq }\mathbb{E}\big[\big|    \sum_{w \in \mathcal{W}^k} \sum_{x \in \mathcal{X}^k} z_1^k(x)  \ID{f^k_t(x,\gamma^k(x),\phi(\mathbf z_1,\boldsymbol \gamma),w)=y}\\
& \qquad \quad \times (\MW(w)- \PW )   \big| \big]+   \Lip{3}_t \Ninf{\mathbf z_1 -  \mathbf z_2}\\
&\substack{{(c)} \\ \leq} \sum_{w \in \mathcal{W}^k} \hspace{-.05cm} \sum_{x \in \mathcal{X}^k}   z^k_1(x) \Exp{| \MW(w) - \PW|} + \Lip{3}_t \Ninf{\mathbf z_1 - \mathbf z_2}\\
& \substack{{(d)} \\ \leq} \BigON+\Lip{3}_t \Ninf{\mathbf z_1 - \mathbf z_2},
\end{align}
where $(a)$ follows from  the triangle inequality and the monotonicity of the expectation operator; $(b)$  follows from  equations~\eqref{eq:relation-models}, \eqref{eq:def-bar-f-k}, \eqref{eq:hat-f-def} and Lemma~\ref{lemma:lipschitz-hat-f-hat-l}; $(c)$ follows from  $\ID{f^k_t(x,\gamma^k(x), \phi(\mathbf z_1,\boldsymbol \gamma), w)=y} \leq~1$,  the triangle inequality and the monotonicity of the expectation operator, and $(d)$ from Lemma~\ref{lemma:square-root-iid}, $z^k_1(x) \leq 1, \forall x \in \mathcal{X}^k,$ and the fact that  the  cardinality  of spaces $\mathcal{X}^k$ and $\mathcal{W}^k$ does not depend on~$|\mathcal{N}^k|$.

\section{Proof of Lemma~\ref{lemma:lipschitz-bar-f-bar-f}}\label{sec:proof_lemma:lipschitz-bar-f-bar-f}

We first show that for any $\mathbf z \in \IX$,  $\boldsymbol \gamma \in \mathcal{G}$, $\MW \in \mathcal{M}_{|\mathcal{N}^k|}(\mathcal{W}^k)$ and $y \in \mathcal{X}^k$ at time  $t \in \mathbb{N}_T$, $\bar f^k_t(\mathbf z, \boldsymbol \gamma,\MW)(y)$ is equal to:
\begin{align}\label{eq:proof-quantization-max}
& \sum_{w \in \mathcal{W}^k} \sum_{x \in \mathcal{X}^k}  z^k(x)     \ID{f^k_t(x,\gamma^k(x), \phi(\mathbf z,\boldsymbol \gamma), w)= y} \MW(w) \nonumber \\
 & \substack{(a) \\ =} \sum_{w \in \breve{\mathcal{W}}^k} \sum_{x \in \mathcal{X}^k}  z^k(x)  \ID{f^k_t(x,\gamma^k(x), \phi(\mathbf z,\boldsymbol \gamma), w)= y} \nonumber\\
  &  \times        \frac{\MW(w)}{\PW} \PW \nonumber\\
  &  \leq \max_{w \in \breve{\mathcal{W}}^k} (\frac{\MW(w)}{\PW} ) \sum_{w \in \breve{\mathcal{W}}^k} \sum_{x \in \mathcal{X}^k} z^k(x) \nonumber\\
  &      \times    \ID{f^k_t(x,\gamma^k(x), \phi(\mathbf z,\boldsymbol \gamma), w)= y}\PW \nonumber \\
    &  \substack{(b) \\ =} \max_{w \in \breve{\mathcal{W}}^k} (\frac{\MW(w)}{\PW} ) \FS^k_t(\mathbf z,\boldsymbol \gamma)(y),
\end{align}
where in $(a)$   set $\breve{\mathcal{W}}^k:=\mathcal{W}^k \backslash \{w \in \mathcal{W}^k | \PW=0\}$ contains all realizations that have  non-zero probability measures (this equality holds a.s.), and $(b)$ follows from \eqref{eq:relation-models},   \eqref{eq:phi-def-function},  \eqref{eq:hat-f-def} and \eqref{eq:proof-joint-mean-field-1}. By similar argument, for any $y \in \mathcal{X}^k$, we have
\begin{align}\label{eq:proof-quantization-min}
\bar f^k_t(\mathbf z, \boldsymbol \gamma,\MW)(y) \geq  \min_{w \in \breve{\mathcal{W}}^k} (\frac{\MW(w)}{\PW} ) \FS^k_t(\mathbf z,\boldsymbol \gamma)(y).
\end{align}
From \eqref{eq:proof-quantization-max} and \eqref{eq:proof-quantization-min},   it results that for   any $\mathbf z_1, \mathbf z_2 \in \IX$ and $\boldsymbol \gamma \in \mathcal{G}$ at time $\NT$, 
\begin{align}
&\mathbb{E}\Ninf{\bar f^k_t(\mathbf z_1,\boldsymbol \gamma, \MW) - \bar f^k_t(\mathbf z_2,\boldsymbol \gamma, \MW)}  \\
& \leq \mathbb{E} \| \hspace{-.1cm}\max_{w \in \breve{\mathcal{W}}^k} \frac{\MW(w)}{\PW}  \FS^k_t(\mathbf z_1,\boldsymbol \gamma) \hspace{-.1cm} - \hspace{-.15cm} \min_{w \in \breve{\mathcal{W}}^k} \frac{\MW(w)}{\PW}\FS^k_t(\mathbf z_2,\boldsymbol \gamma) \|\\
& \substack{(c)\\ \leq } \mathbb{E} \| \hspace{-.1cm}\max_{w \in \breve{\mathcal{W}}^k} \frac{\MW(w)}{\PW} \FS^k_t(\mathbf z_1,\boldsymbol \gamma) \hspace{-.1cm} - \hspace{-.15cm} \max_{w \in \breve{\mathcal{W}}^k} \frac{\MW(w)}{\PW}\FS^k_t(\mathbf z_2,\boldsymbol \gamma) \|   \\
& \qquad +  \mathbb{E} \| (\max_{w \in \breve{\mathcal{W}}^k} \frac{\MW(w)}{\PW} - \min_{w \in \breve{\mathcal{W}}^k} \frac{\MW(w)}{\PW})\FS^k_t(\mathbf z_2,\boldsymbol \gamma) \|\\
&= \Exp{\max_{w \in \breve{\mathcal{W}}^k} \frac{\MW(w)}{\PW}} \Ninf{\FS^k_t(\mathbf z_1,\boldsymbol \gamma) - \FS^k_t(\mathbf z_2,\boldsymbol \gamma) }\\
&\qquad +  \Exp{\max_{w \in \breve{\mathcal{W}}^k} \frac{\MW(w)}{\PW} - \min_{w \in \breve{\mathcal{W}}^k} \frac{\MW(w)}{\PW}} \Ninf{\FS^k_t(\mathbf z_2,\boldsymbol \gamma)}\\
&\substack{(d)\\ =}    \Ninf{\FS^k_t(\mathbf z_1,\boldsymbol \gamma) - \FS^k_t(\mathbf z_2,\boldsymbol \gamma) } \substack{(e)\\ \leq }  \Lip{3}_t \Ninf{ \mathbf z_1 - \mathbf z_2},
\end{align}
where  $(c)$  follows from the triangle inequality and the monotonicity of the expectation operator, and $(d)$ holds  for every $ w  \in \breve{\mathcal{W}}^k$,
$
\mathbb{E}[\frac{\MW(w)}{\PW}] =\frac{\Exp{\sum_{i \in \mathcal{N}^k} \ID{w^i_t=w}}}{|\mathcal{N}^k|\PW} 
=  \frac{ |\mathcal{N}^k| \PW}{|\mathcal{N}^k|\PW}=1,
$
and $(e)$ follows from Lemma~\ref{lemma:lipschitz-hat-f-hat-l}.

%
\section{Proof of Lemma~\ref{proposition:quantization-basic}}\label{sec:proof_proposition:quantization-basic}

The proof proceeds by  backward induction. At   $t=T$,  
\begin{align*}
V^1_T(\mathbf z^1)&= \min_{\boldsymbol \gamma^1\in \mathcal{G}} \LS_T(\mathbf z^1,\boldsymbol \gamma^1)\\
& \substack{(a) \\ \leq } \min_{\boldsymbol \gamma^1 \in \mathcal{G}} \left( \big|\LS_T(\mathbf z^1,\boldsymbol \gamma^1) -\LS_T(\mathbf z^2,\boldsymbol \gamma^1) \big|+\LS_T(\mathbf z^2,\boldsymbol \gamma^1) \right)\\
& \substack{(b) \\ \leq } \min_{\boldsymbol \gamma^1 \in \mathcal{G}} \left( \Lip{4}_T \| \mathbf z^1- \mathbf z^2\|+\LS_T(\mathbf z^2,\boldsymbol \gamma^1) \right)\\
&\substack{(c) \\ =} \Lip{4}_T \|\mathbf z^1-\mathbf z^2\|+ \min_{\boldsymbol \gamma^2 \in \mathcal{G}} \LS_T(\mathbf z^2,\boldsymbol \gamma^2) \\
&\substack{(d) \\ =} \Lip{4}_T \|\mathbf  z^1-\mathbf z^2\|+ V^2_T(\mathbf z^2),
\end{align*}
where $(a)$ follows from the triangle inequality,  the fact that $\LS_T(\boldsymbol \cdot ) \in \mathbb{R}_{\geq 0}$, and  the monotonicity of the minimum operator; $(b)$ follows from Lemma~\ref{lemma:lipschitz-hat-f-hat-l} and the monotonicity of the minimum operator; $(c)$  follows from the fact that the space of minimization~$\mathcal{G}$ is the same for both $\boldsymbol \gamma^1$ and $\boldsymbol \gamma^2$, and $(d)$ follows from~\eqref{eq:prop-D-t}. Therefore, \eqref{eq:prop-D-difference} holds at $t=T$, where $\Lip{5}_T=\Lip{4}_T$ and $\Lip{6}_T=0$. Now, assume that~\eqref{eq:prop-D-difference} holds at time $t+1$, i.e.,  
$| V^1_{t+1} (\mathbf z^1) -V^2_{t+1} (\mathbf z^2)| \leq \Lip{5}_{t+1}\Ninf{ \mathbf z^1 - \mathbf z^2} +\Lip{6}_{t+1} \Delta.$
It is desired to  prove \eqref{eq:prop-D-difference} for time $t$. To this end, the following relations are derived:
\begin{align*}
V^1_t(\mathbf z^1)&= \min_{\boldsymbol \gamma^1 \in \mathcal{G}} \big(\LS_t(\mathbf z^1,\boldsymbol \gamma^1)+\Exp{ V^1_{t+1}( f^1_t(\mathbf z^1, \boldsymbol \gamma^1,\mathbf d_t))}\big)\\
&\substack{(e) \\ \leq} \min_{\boldsymbol \gamma^1 \in \mathcal{G}}\Big( \big| \LS_t(\mathbf z^1,\boldsymbol \gamma^1) - \LS_t(\mathbf z^2,\boldsymbol \gamma^1)\big| \\
&  \quad + \Exp{\big|  V^1_{t+1}( f^1_t(\mathbf z^1, \boldsymbol \gamma^1,\mathbf d_t)) -V^2_{t+1}( f^2_t(\mathbf z^2, \boldsymbol \gamma^1,\mathbf d_t))     \big|}\\
& \quad +\LS_t(\mathbf z^2,\boldsymbol \gamma^1)+\Exp{V^2_{t+1}(f^2_t(\mathbf z^2, \boldsymbol \gamma^1,\mathbf d_t))} \Big)\\
&\substack{(f) \\ \leq}(\Lip{4}_t+\Lip{5}_{t+1}\Lip{z}_t)\|\mathbf z^1 -\mathbf z^2 \| + (\Lip{5}_{t+1}+\Lip{6}_{t+1}) \Delta\\
&\quad  + \min_{\boldsymbol \gamma^2 \in \mathcal{G}} \Big(\LS_t(\mathbf z^2,\boldsymbol \gamma^2)+\Exp{V^2_{t+1}(f^2_t(\mathbf z^2, \boldsymbol \gamma^2,\mathbf d_t))}  \Big)\\
&\substack{(g) \\ =}\Lip{5}_{t}\|\mathbf z^1 -\mathbf z^2 \| + \Lip{6}_{t} \Delta + V^2_t(\mathbf z^2),
\end{align*}
where $(e)$ follows from the triangle inequality,  the fact that $\LS_t (\boldsymbol \cdot), V^j_{t+1}(\boldsymbol \cdot) \in \mathbb{R}_{ \geq 0}$, $j \in \{1,2\}$, and the monotonicity of the expectation and minimum  operators; $(f)$ follows from Lemma~\ref{lemma:lipschitz-hat-f-hat-l}, equations \eqref{eq:prop-z1-z2},  the induction assumption at time $t+1$,  the monotonicity of the expectation and minimum  operators, and the fact  that the space~$\mathcal{G}$ is the same for both $\boldsymbol \gamma^1$ and $\boldsymbol \gamma^2$, and  $(g)$ follows  from \eqref{eq:prop-D-t} and~\eqref{eq:H5H6}.

\begin{IEEEbiography}[{\includegraphics[width=1in,height=1.25in,clip,keepaspectratio]{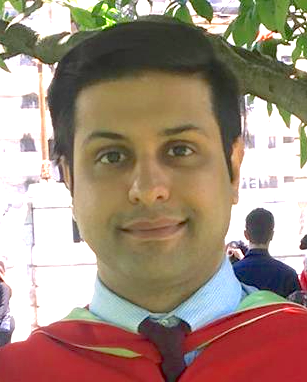}}]{Jalal Arabneydi}
received the Ph.D. degree  in Electrical and Computer Engineering from
  McGill University, Montreal, Canada in 2017.  
   He is currently a postdoctoral
  fellow at Concordia University. He  was  the recipient of the
best student paper award at the 53rd Conference on
Decision and Control (CDC), 2014. His  principal research interests
  include  stochastic control,  robust optimization, game theory,  large-scale system, multi-agent reinforcement learning  with applications in complex networks including  smart grids, swarm robotics,  and finance.  His current research interest is  focused on what he calls deep planning, which bridges  decision making theory and artificial intelligence. The ultimate goal is to define proper mathematical tools and solution concepts in order to develop large-scale decision-making algorithms that work under imperfect information and incomplete knowledge with analytical performance guarantees.
\end{IEEEbiography}
\begin{IEEEbiography}
[{\includegraphics[width=1in,height=1.25in,clip,keepaspectratio]{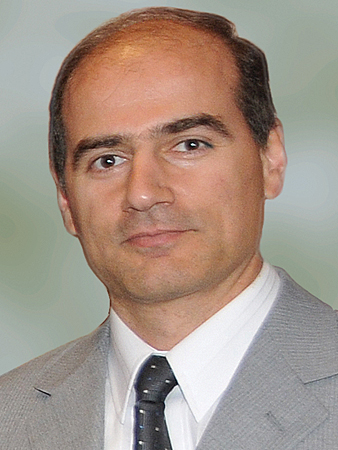}}]{Amir G. Aghdam}
 received the Ph.D. degree in electrical and computer engineering from the University of Toronto, Toronto, ON, Canada, in 2000, and is currently a Professor in the Department of Electrical and Computer Engineering at Concordia University, Montreal, Canada. He is a member of Professional Engineers Ontario, chair of the Conference Editorial Board of IEEE Control Systems Society, Editor-in-Chief of the IEEE Systems Journal, and has served as an Associate Editor of the IEEE Transactions on Control Systems Technology, IEEE Access, IET Control Theory \& Applications, and the Canadian Journal of Electrical and Computer Engineering. He has been a member of the Technical Program Committee of a number of conferences, including the IEEE Conference on Systems, Man and Cybernetics (IEEE SMC) and the IEEE Multiconference on Systems and Control (IEEE MSC). He has served as a member of the Review Panel/Committee for the NSF, Italian Research and University Evaluation Agency (ANVUR), Innovation Fund Denmark Projects, and the Natural Sciences and Engineering Research Council of Canada (NSERC) ECE Evaluation Group. Dr. Aghdam was the 2014–2015 President of IEEE Canada and Director (Region 7), IEEE, Inc., and was also a member of the IEEE Awards Board for this period. He was a Visiting Scholar at Harvard University in fall 2015, and was an Associate at the Harvard John A. Paulson School of Engineering and Applied Sciences from September 2015 to December 2016. His research interests include multi-agent networks, distributed control, optimization and sampled-data systems. Prof.  Aghdam is a recipient of the 2009 IEEE MGA Achievement Award, and is currently a member of the IEEE Medal of Honor Committee Award and 2020 IEEE Canada J. M. Ham Outstanding Engineering Educator Award. Dr. Aghdam was a member of the IEEE Medal of Honor Committee for 2017-2019 and is currently the Vice-Chair of the IEEE Medals Council.
\end{IEEEbiography}

\clearpage

\section*{\textbf{Partially exchangeable agents}~\cite[Chapter 2]{arabneydi2016new}}

Consider a discrete-time control system consisting of a finite population of  agents, where agents are partitioned into $K \in \mathbb{N}$ disjoint sub-populations.   Denote by $\mathcal{K}$ the set of sub-populations,  by $\mathcal{N}^k$  the agents of sub-population $k \in \mathcal{K}$, and by  $\mathcal{N}$ the entire population of agents; note that $\mathcal{N}=\cup_{k \in \mathcal{K}} \mathcal{N}^k$.  Given the control horizon $T \in \mathbb{N}$, let the state, action, and the noise of agent $i \in \mathcal{N}^k$ of sub-population $k \in \mathcal{K}$ at time $t \in \mathbb{N}_T$  be denoted by  $x^i_t \in \mathcal{X}^k$, $u^i_t \in \mathcal{U}^k$, and $w^i_t \in \mathcal{W}^k$, respectively.  For the entire population,  the joint state, joint action, and joint noise are analogously denoted by $\mathbf x_t=(x^i_t)_{i \in \mathcal{N}} \in \mathcal{X}$, $\mathbf u_t=(u^i_t)_{i \in \mathcal{N}} \in \mathcal{U}$,  and $\mathbf w_t=(w^i_t)_{i \in \mathcal{N}} \in \mathcal{W}$  at time $t \in \mathbb{N}_T$. 

For any  $k \in \mathcal{K}$, the initial state of agent $i$ of sub-population~$k$ is denoted by  $x^i_1 \in \mathcal{X}^k$, and at time $t \in \mathbb{N}_T$ its state evolves as follows:
\begin{equation}\label{eq:basic-model-f-pre}
x^i_{t+1}=f^i_t(\mathbf x_t,\mathbf{u}_t, w^i_t),
\end{equation}
where $f_t^i: \mathcal{X} \times \mathcal{U} \times \mathcal{W}^k \rightarrow \mathcal{X}^k$ describes the dynamics of agent $i$ at time $t \in \mathbb{N}_T$.  The primitive random variables $\{\mathbf x_1, \mathbf w_1,\ldots,\mathbf w_T\}$ are defined on a common probability space. In short, the dynamics of the entire  population may be expressed   in an augmented form as: 
\begin{equation}
\mathbf x_{t+1}=f_t(\mathbf x_t,\mathbf{u}_t, \mathbf w_t).
\end{equation}
 A per-step cost $c_t(\mathbf{x}_t,\mathbf u_t):\mathcal{X}\times \mathcal{U} \rightarrow \mathbb{R}_{\geq 0}$ is incurred at any time $t \in \mathbb{N}_T$, which reflects the desirable behaviour of the system  at time $t$. Let $I^i_t  \subseteq \{\mathbf x_{1:t}, \mathbf u_{1:t-1} \}$ denote the information of agent $i$ at time $t \in \mathbb{N}_T$ for any  $i \in \mathcal{N}$. In particular,  the action of agent $i$ at time $t$ is given by 
\begin{equation}\label{eq:control-law-general-pre}
u^i_t=g^i_t(I^i_t),
\end{equation}
where deterministic function $g^i_t$  is called the  control law of agent $i$ at time $t \in \mathbb{N}_T$. The set of control laws $\mathbf{g}:=\{(g^i_t)_{i \in \mathcal{N}}\}_{t=1}^T$ is defined as the \emph{strategy} of  the system and  the control  performance is  described by
\begin{equation}\label{eq:J-general-def-pre}
J(\mathbf{g})= \mathbb{E}^{\mathbf{g}} \big[ \sum_{t=1}^T c_t(\mathbf x_t,\mathbf{u}_t)\big],
\end{equation}
where the expectation is taken with respect to the probability measures induced by the choice of the strategy $\mathbf g$.

\begin{Definition}[\textbf{Partially exchangeable agents}]\label{def:Exchangeable agents}
The multi-agent system described above is said to be a system with partially exchangeable agents if the following two conditions hold  for  any pair of agents  $(i,j) \in \mathcal{N}^k$ of any sub-population $k \in \mathcal{K}$ at any time $t \in \mathbb{N}_T$:
\begin{enumerate}
\item   $
 \sigma_{i,j} \big(f_t(  \mathbf x_t, \mathbf u_t, \mathbf w_t)\big)= f_t\big(\sigma_{i,j} (\mathbf x_t),\sigma_{i,j} ( \mathbf u_t),\sigma_{i,j}  (\mathbf w_t)\big),
$
 i.e., exchanging agents $i$ and $j$ has no effect on the system dynamics.
\item [2)] $
 c_t(  \mathbf x_t, \mathbf u_t)= c_t(\sigma_{i,j}  (\mathbf x_t),\sigma_{i,j}  (\mathbf u_t)),$
 i.e., exchanging  agents $i$ and $j$ does not impact the cost.
\end{enumerate}
\end{Definition}
\begin{Remark}\label{remark:exch-counter-pre}
\emph{Note that the notion of  exchangeable agents is different from the notion of  exchangeable random variables often used in the probability literature (e.g., de Finetti's Theorem).  See a counterexample in~\cite[Chapter 2]{arabneydi2016new} where it is illustrated that neither is necessarily implied by the other.}
\end{Remark}

 \subsection{Empirical distribution}
Let $\mathbf b=(b^1,\ldots,b^n)$ denote a vector of $n \in \mathbb{N}$  samples from set $\mathcal{B}:=\{a_1,\ldots,a_{|\mathcal{B}|}\}$, where $b^i \in \mathcal{B}, i \in \mathbb{N}_n $. The empirical distribution function $\xi: \prod_{i=1}^n \mathcal{B} \rightarrow \mathcal{E}_n(\mathcal{B})$  is defined as  a real-valued vector of size $|\mathcal{B}|$ such that
\begin{equation}\label{eq:def-Emp-pre}
\xi(\mathbf{b})(a_j)= \frac{1}{n}\sum_{i=1}^n \mathds{1}(b^i=a_j), \quad j \in \mathbb{N}_{|\mathcal{B}|}.
\end{equation}
\begin{Lemma}\label{lemma-exchangeable-basic-empirical-pre}
Let $h$  denote an arbitrary  exchangeable function over the product space $\prod_{i=1}^n \mathcal{B}$.  Then, there  exists a function~$\bar h$ such that
 $h(\mathbf b)=\bar h(\xi(\mathbf b)).$
\end{Lemma}
\begin{proof}
Define vector  $\mathbf{s} \in \prod_{i=1}^n \mathcal{B}$ such that  its $i$-th element, $i \in \mathbb{N}_n$, is  equal to $a_j \in \mathcal{B}$,    where index $j \in \mathbb{N}_{|\mathcal{B}|}$    satisfies the following inequalities:  $ \ID{j \neq 1} n\sum_{l=1}^{j-1}\xi(\mathbf b)(a_l)<i\leq n\sum_{l=1}^{j}\xi(\mathbf b)(a_l)$.  As a result, vector $ \mathbf s$ is an exchanged  version of vector $\mathbf b$ and it  is completely determined  by  $\xi(\mathbf b)$.  Since  $h(\mathbf b)$ is exchangeable, we get   $h(\mathbf b)=h(\mathbf s)=:\bar h(\Emp{\mathbf b)}$. 
\end{proof}
\begin{figure}
\centering
\vspace{-0cm}
\scalebox{0.9}{
\includegraphics[trim={0cm 0cm 0 0cm},clip, width=\linewidth]{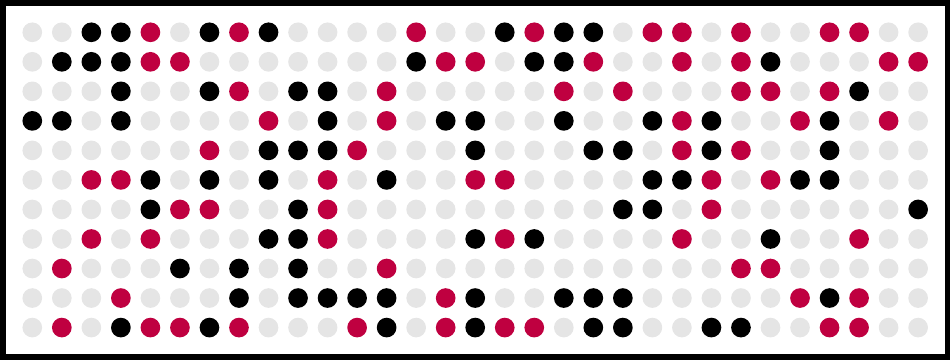}}
\caption{\edit{Partially exchangeable agents  wherein the system is invariant to permutation in each sub-population~\cite[Chapter 2]{arabneydi2016new}.}}
\end{figure}
\subsection{Structure of  deep teems with partially exchangeable agents}
The short-hand notation $\sigma(\mathbf x)$ is used to denote any arbitrary permuted version of vector $\mathbf x$.
\begin{Proposition}\label{prop:model-mean-field-pre}
For any controlled Markov chain system with partially exchangeable agents  described in Definition~\ref{def:Exchangeable agents}, there exist functions $(f^k_t)_{k \in \mathcal{K}}$ and $\bar c_t$, $t \in \mathbb{N}_T$, such that the dynamics of agent $i $ of sub-population $k \in \mathcal{K}$  can be written~as
\begin{equation}\label{eq:dynamics-f-mean-field-pre}
x^i_{t+1}=f^k_t(x^i_t,u^i_t, \boldsymbol{\mathfrak{D}}_t, w^i_t),
\end{equation}
and the per-step cost at time $t \in \mathbb{N}_T$,  can  be expressed as
\begin{equation}\label{eq:per-step-cost-mean-field-pre}
 \bar c_t(\boldsymbol{\mathfrak{D}}_t)=c_t(\mathbf{x}_t, \mathbf u_t).
\end{equation}
\end{Proposition}
\begin{proof}
Consider  agent  $i \in \mathcal{N}^k$ of sub-population $k \in \mathcal{K}$. Let $\mathbf x^{-i}_t$ and $\mathbf u^{-i}_t$ respectively denote the  joint state $\mathbf x_t$ and joint action $\mathbf u_t$ of all agents except agent $i$. The dynamics given by~\eqref{eq:basic-model-f-pre} can be rewritten as follows:
\begin{align}\label{eq:proof-dynamic-ij-pre}
x^i_{t+1}&=f^i_t({x^i_t,u^i_t, \mathbf x^{-i}_t,\mathbf u^{-i}_t,w^{i}_t}).
\end{align}
Let $\boldsymbol \sigma^{-i}$ denote any arbitrary permutation of  the entire population excluding agent $i$, where the permutation is allowed only within the sub-populations.   According to condition 1 of Definition~\ref{def:Exchangeable agents}, arbitrarily  permuting   the state and action of other  agents  will not change the dynamics of agent $i \in \mathcal{N}^k$. Thus,
\[f^i_t({x^i_t,u^i_t, \mathbf x^{-i}_t,\mathbf u^{-i}_t,w^{i}_t})=f^i_t({x^i_t,u^i_t, \boldsymbol \sigma^{-i} (\mathbf x^{-i}_t, \mathbf u^{-i}_t),w^{i}_t}).\]
By applying Lemma~\ref{lemma-exchangeable-basic-empirical-pre} to all sub-populations successively, it results that  there exists a function $\bar{f}^i$ such that
\begin{equation}\label{eq:proof-exchangeable-proposition-pre}
f^i_t({x^i_t,u^i_t, \mathbf x^{-i}_t,\mathbf u^{-i}_t,w^{i}_t})=\bar f^i_t(x^i_t,u^i_t,\xi( \mathbf x^{-i}_t,\mathbf u^{-i}_t), w^i_t),
\end{equation}
where, by  a slight abuse of notation,  $\Emp{\mathbf x^{-i}_t,\mathbf u^{-i}_t}$ denotes a vector consisting of the empirical distributions  of  states and actions of agents, except  agent~$i$,  in all sub-populations. Note that $\Emp{\mathbf x^{-i}_t,\mathbf u^{-i}_t}$  can be identified  by   $\xi( \mathbf x_t,\mathbf u_t)=\boldsymbol{\mathfrak{D}}_t$ and $(x^i_t,u^i_t)$.  Thus, from equation \eqref{eq:proof-exchangeable-proposition-pre} there exists a function $\tilde f_t$ such that
\[f^i_t({x^i_t,u^i_t, \mathbf x^{-i}_t,\mathbf u^{-i}_t,w^{i}_t})=\tilde f^i_t(x^i_t,u^i_t,\boldsymbol{\mathfrak{D}}_t, w^i_t).\]
Now, consider two arbitrary agents $i$ and $j$ of sub-population~$k$. From condition 1 of  Definition~\ref{def:Exchangeable agents}, we have
\[x^i_{t+1}= \tilde f^i_t(x^i_t,u^i_t,\boldsymbol{\mathfrak{D}}_t, w^i_t)=\tilde f^j_t(x^i_t,u^i_t,\boldsymbol{\mathfrak{D}}_t, w^i_t),\]
and
\[x^j_{t+1}= \tilde f^j_t(x^j_t,u^j_t,\boldsymbol{\mathfrak{D}}_t, w^j_t)=\tilde f^i_t(x^j_t,u^j_t,\boldsymbol{\mathfrak{D}}_t, w^j_t).\]
Hence, $\tilde{f}^i_t=\tilde{f}^j_t=:f^k_t,\hspace{.1cm} \forall i,j \in \mathcal{N}^k$.

A similar argument applies to the per-step cost. Let  $\boldsymbol \sigma$ denote any  arbitrary permutation of the entire population, where the permutation is allowed only within the sub-populations. Condition~2 of Definition~\ref{def:Exchangeable agents},  implies (directly) that arbitrarily permuting the agents of sub-population $k \in \mathcal{K}$ does not change the cost. Thus,  by applying the result  of Lemma~\ref{lemma-exchangeable-basic-empirical-pre} to all sub-populations $k \in \mathcal{K}$, there exists a function $\bar c_t$ such that 
\[ c_t(\mathbf x_t,\mathbf u_t)=c_t(\boldsymbol \sigma(\mathbf x_t,\mathbf u_t))=:\bar c_t(\xi (\mathbf x_t,\mathbf u_t))=\bar c_t(\boldsymbol{\mathfrak{D}}_t).\] 
\end{proof}

\end{document}